\documentclass[11pt,reqno]{amsart}
\usepackage{cmap,mathtools}
\usepackage[T1]{fontenc}
\usepackage[utf8]{inputenc}
\usepackage[english]{babel}
\usepackage{amsfonts,amssymb,amsthm,amsmath}
\usepackage{url}
\usepackage{enumitem}
\usepackage{float}

\usepackage{secdot}

\allowdisplaybreaks

\usepackage{graphicx}
\usepackage{epstopdf}
\usepackage{a4wide}
\usepackage{xcolor}
\usepackage[colorlinks=true,linktocpage,pdfpagelabels,
bookmarksnumbered,bookmarksopen]{hyperref}
\definecolor{ForestGreen}{rgb}{0.1,0.6,0.05}
\definecolor{EgyptBlue}{rgb}{0.063,0.1,0.6}
\hypersetup{
	colorlinks=true,
	linkcolor=EgyptBlue,         
	citecolor=ForestGreen,
	urlcolor=olive
}

\usepackage[hyperpageref]{backref}

\newtheorem{theorem}{Theorem}
\newtheorem{proposition}[theorem]{Proposition}
\newtheorem{lemma}[theorem]{Lemma}
\newtheorem{corollary}[theorem]{Corollary}

\theoremstyle{definition}
\newtheorem{definition}[theorem]{Definition}
\newtheorem{remark}[theorem]{Remark}

\let\OLDthebibliography\thebibliography
\renewcommand\thebibliography[1]{
	\OLDthebibliography{#1}
	\setlength{\parskip}{1pt}
	\setlength{\itemsep}{1pt plus 0.3ex}
}

\numberwithin{equation}{section}
\numberwithin{theorem}{section}
\numberwithin{equation}{section}
\numberwithin{theorem}{section}

\newcommand{\N}{\mathbb{N}}

\newcommand{\al} {\alpha}
\newcommand{\be} {\beta}
\newcommand{\ga}{\gamma}

\newcommand{\Om}{\Omega}
\newcommand{\la}{\lambda}

\def\inpr#1{\left\langle #1\right\rangle}

\title[Szeg\H{o}-Weinberger type inequalities]{Szeg\H{o}-Weinberger type inequalities for symmetric domains with holes}

\author[T.V.~Anoop]{T.V.~Anoop}
\author[V.~Bobkov]{Vladimir Bobkov}
\author[P.~Dr\'abek]{Pavel Dr\'abek}

\address[T.V.~Anoop]{\newline\indent
	Department of Mathematics,
	Indian Institute of Technology Madras, Chennai 36, India
}
\email{anoop@iitm.ac.in}

\address[V.~Bobkov]{
	\newline\indent 
	Institute of Mathematics, Ufa Federal Research Centre, RAS,
	\newline\indent 
	Chernyshevsky str. 112, 450008 Ufa, Russia
	\newline\indent
	Department of Mathematics and NTIS, Faculty of Applied Sciences,
	\newline\indent 
	University of West Bohemia, Univerzitn\'i 8, 301 00 Plze\v{n}, Czech Republic
}
\email{bobkov@matem.anrb.ru}

\address[P.~Dr\'abek]{\newline\indent
	Department of Mathematics and NTIS, Faculty of Applied Sciences,
	\newline\indent 
	University of West Bohemia, Univerzitn\'i 8, 301 00 Plze\v{n}, Czech Republic
}
\email{pdrabek@kma.zcu.cz}

\date{}


\begin{document}
	\begin{abstract}
		Let $\mu_2(\Omega)$ be the first positive eigenvalue of the Neumann Laplacian in a bounded domain $\Omega \subset \mathbb{R}^N$.
		It was proved by Szeg\H{o} for $N=2$ and by Weinberger for $N \geq 2$ that among all equimeasurable domains $\mu_2(\Omega)$ attains its global maximum if  $\Omega$ is a ball. 
		In the present work, we develop the approach of Weinberger in two directions. 
        Firstly, we refine the Szeg\H{o}-Weinberger result for a class of domains of the form $\Omega_{\text{out}} \setminus \overline{\Omega}_{\text{in}}$ which are either centrally symmetric or symmetric of order $2$ (with respect to every coordinate plane $(x_i,x_j)$)  by showing that $\mu_{2}(\Omega_{\text{out}} \setminus \overline{\Omega}_{\text{in}}) \leq \mu_2(B_\beta \setminus \overline{B}_\alpha)$, where $B_\al, B_\be$ are balls centered at the origin such that $B_\al \subset \Om_{\text{in}}$ and  $|\Om_{\text{out}}\setminus \overline{\Om}_{\text{in}}|=|B_\be\setminus  \overline{B}_\al|$. 
		Secondly, we provide Szeg\H{o}-Weinberger type inequalities for higher eigenvalues by imposing additional symmetry assumptions on the domain.
		Namely, if $\Omega_{\text{out}} \setminus \overline{\Omega}_{\text{in}}$ is symmetric of order $4$, then we prove $\mu_{i}(\Omega_{\text{out}} \setminus \overline{\Omega}_{\text{in}}) \leq \mu_i(B_\beta \setminus \overline{B}_\alpha)$ for $i=3,\dots,N+2$, where we also allow $\Om_{\text{in}}$ and  $B_\al$ to be empty.
		If $N=2$ and the domain is symmetric of order $8$, then the latter inequality persists for $i=5$.
		Counterexamples to the obtained inequalities for domains outside of the considered symmetry classes are given.
		The existence and properties of  nonradial domains with required symmetries in higher dimensions are discussed.
		As an auxiliary result, we obtain the non-radiality of the eigenfunctions associated to $\mu_{N+2}(B_\beta \setminus \overline{B}_\alpha)$.
	\end{abstract} 
	\maketitle

	\setcounter{tocdepth}{1}
	\tableofcontents
	
	\section{Introduction}\label{sec:introduction}

We consider the Neumann eigenvalue problem
\begin{equation}\label{eq:N}
	\tag{$\mathcal{EP}$}
	\left\{
	\begin{aligned}
		-\Delta u &= \mu u &&\text{in}~\Omega,\\
		\frac{\partial u}{\partial n} &= 0 &&\text{on}~\partial \Omega,
	\end{aligned}
	\right.
\end{equation}
where $\Omega$ obeys the following general assumption:
\begin{enumerate}[label={$\mathbf{(A_1)}$}]
	\item\label{assumption0}
	$\Omega 
	\subset \mathbb{R}^N$, $N \geq 2$, is a bounded domain with boundary $\partial \Omega$, such that the embedding $H^1(\Omega) \hookrightarrow L^2(\Omega)$ is compact. 
\end{enumerate}
As a model case, one can ask $\Omega$ to be smooth or Lipschitz.
Under the assumption \ref{assumption0}, the spectrum of \eqref{eq:N} consists of a discrete sequence of  eigenvalues
$$
0 = \mu_1(\Omega) < \mu_2(\Omega) \leq \mu_3(\Omega) \leq \dots 
$$

The Szeg\H{o}-Weinberger inequality states that
\begin{equation}\label{eq:SW}
	\mu_2(\Omega) \leq \mu_2(B),
\end{equation}
where $B$ is an open $N$-ball of the same Lebesgue measure as $\Omega$,
and if equality holds in \eqref{eq:SW}, then $\Omega=B$ up to a set of zero measure.
The inequality \eqref{eq:SW} was conjectured for $N=2$ by \textsc{Kornhauser \& Stakgold} \cite{KS} who established that the disk is a local maximizer of $\mu_2(\Omega)$. 
Later, \eqref{eq:SW} was obtained by \textsc{Szeg\H{o}} \cite{szego} for regular planar domains bounded by a simple closed curve (Jordan domains, for short).
Then, \textsc{Weinberger} \cite{weinberger} proved \eqref{eq:SW} in the general higher-dimensional case without any topological restrictions on $\Omega$. 
Qualitative versions of \eqref{eq:SW} were investigated by \textsc{Nadirashvili} \cite{nad1} in the case $N=2$ (for Jordan domains), and by \textsc{Brasco \& Pratelli} \cite{BP} in the general case $N \geq 2$, see also the survey \cite{BDP}.

Notice that $\mu_2(B)$ is explicitly given by
\begin{equation*}\label{eq:SWexp}
	\mu_2(B) = \left(\frac{\omega_N}{|\Omega|}\right)^\frac{2}{N} \left(p_{\frac{N}{2},1}^{(1)}\right)^2,
\end{equation*}
where $\omega_N$ is the volume of a unit $N$-ball, and, following the terminology of \cite{LS}, $p_{\nu,j}^{(l)}$ stands for the $j$-th positive zero of the function $(r^{1-\nu} J_{\nu+l-1}(r))'$,  $J_\nu$ being the Bessel function of the first kind of order $\nu$.

In the planar case $N=2$, various improvements of the Szeg\H{o}-Weinberger inequality are known 
under additional assumptions on the symmetry of $\Omega$. 
We say that a domain $\Omega \subset \mathbb{R}^2$ is \textit{symmetric of order $q \in \mathbb{N}$} if, after an appropriate translation, $\Omega$ is invariant under the rotation by angle $2\pi/q$.
It was proved by \textsc{Hersch} in \cite[Section 5.3]{hersh2} that for any Jordan domain $\Omega$ symmetric of order $q \geq 3$ the following inequality holds:
\begin{equation}\label{eq:SWsym}
	\mu_3(\Omega) \leq \mu_2(B) = \frac{\pi}{|\Omega|} \left(p_{1,1}^{(1)}\right)^2
	\approx \frac{10.6499}{|\Omega|}.
\end{equation}
In fact, to obtain \eqref{eq:SWsym} one can notice that $\mu_2(\Omega)=\mu_3(\Omega)$ for such class of domains, see  \cite[Lemma 4.1]{AB}. 
Moreover, \textsc{Ashbaugh \& Benguria} showed in \cite[Theorem 4.3]{AB} that \eqref{eq:SWsym} is satisfied for all bounded domains (regardless the topology) which are symmetric of order $4$.
We refer the reader to \textsc{Enache \& Philippin} \cite{EF1,EF2} for further inequalities involving $\mu_k(\Omega)$ for domains with symmetry of order $q \geq 2$.
\textsc{Hersch} in \cite[Section 5.4]{hersh2} also proved that for any Jordan domain $\Omega$ symmetric of order $4$ there holds
\begin{equation}\label{eq:SWsym4}
\mu_4(\Omega) \leq \mu_4(B) = \frac{\pi}{|\Omega|} \left(p_{1,1}^{(2)}\right)^2
\approx \frac{29.3059}{|\Omega|}.
\end{equation}
Let us mention that for a general bounded domain $\Omega \subset \mathbb{R}^N$ the inequality
\begin{equation}\label{eq:GNP}
	\mu_3(\Omega) \leq 2^\frac{2}{N}\mu_2(B)
\end{equation}
was established by 
\textsc{ Girouard,  Nadirashvili, \&  Polterovich} 
\cite{GNP} for $N=2$ (for Jordan domains), and by \textsc{Bucur \& Henrot} \cite{BH} for all $N\geq 2$.
Notice that if equality holds in \eqref{eq:GNP}, then $\Omega$ is a.e.\ a union of two disjoint equimeasurable balls.

\medskip
The main aim of the present work is to generalise the inequalities \eqref{eq:SW}, \eqref{eq:SWsym}, and \eqref{eq:SWsym4} in two directions: to the higher-dimensional case, and to domains with ``holes''.
Moreover, we present an inequality which generalises \eqref{eq:SW}, \eqref{eq:SWsym}, and \eqref{eq:SWsym4} to domains which are symmetric of order $8$.
First, let us introduce the following natural generalisation of the notion of symmetry of order $q$ to higher dimensions, cf.\ \cite[Section 4]{AB}.
\begin{definition}\label{def}
A domain $\Omega \subset \mathbb{R}^N$ is \textit{symmetric of order $q$}
if there exists an isometry $T$ 
such that
$R^{2\pi/q}_{i,j} T(\Omega) = T(\Omega)$ for any $1 \leq i < j \leq N$, where $R^{2\pi/q}_{i,j}$ denotes the rotation (in the anticlockwise direction with respect to the origin) by angle $2\pi/q$ in the coordinate plane $(x_i,x_j)$.
\end{definition}

We will also use the following slight variation on the classical notion of central symmetry.
\begin{definition}\label{def2}
A domain $\Omega \subset \mathbb{R}^N$ is \textit{centrally symmetric} if there exists a vector $\zeta \in \mathbb{R}^N$ such that $x \in \Omega+\zeta$ if and only if $-x \in \Omega+\zeta$.
\end{definition}

\begin{remark}
    Since the action of an isometry on $\Om$ does not change the value of $\mu_k(\Omega)$, from now on, except of Section \ref{sec:existence}, we will always assume that the isometry $T$ in Definition \ref{def} is the identity, and the translation $\zeta$ in Definition \ref{def2} is the zero vector.
	On the other hand, the presence of $T$ and $\zeta$ will be important in Section \ref{sec:existence} to show the relation between different symmetry classes.
\end{remark}

\begin{remark}
In the planar case $N=2$, the symmetry of order $2$ is equivalent to the central symmetry.
When $N\ge 4$ is an even dimension, the symmetry of order $2$ always implies the central symmetry, but not vice versa.
When $N \geq 3$ is an odd dimension, these two  notions are independent. 
We refer to Section \ref{subsec:2} for a detailed discussion.
\end{remark}

Let us now characterise a class of domains with ``holes'' by the following assumption:
\begin{enumerate}[label={$\mathbf{(A_2)}$}]
	\item\label{assumption}
	$\Omega = \Omega_{\text{out}} \setminus \overline{\Omega}_{\text{in}}$ is a domain in $\mathbb{R}^N$, where the domain $\Omega_{\text{in}}$ is compactly contained in the domain $\Omega_{\text{out}}$.
	If $\Om_{\text{in}}$ is nonempty, then we additionally assume $0 \in \Omega_{\text{in}}$.
\end{enumerate}

\begin{remark}
A domain $\Omega$ satisfying \ref{assumption} might possess other ``holes'' except $\Omega_{\text{in}}$, or might possess no ``holes'' at all, in which case $\Omega_{\text{in}} = \emptyset$.
The assumption $0 \in \Omega_{\text{in}}$ is imposed in order to guarantee that if $\Omega$ is symmetric of order $q$ or centrally symmetric, and $\Omega_{\text{in}}$ is nonempty, then $\Omega_{\text{in}}$ contains a ball centred at the origin.
\end{remark}
Throughout this paper, $B_\gamma$ will stand for the open ball of radius $\gamma > 0$ centred at the origin. For $\gamma=0$, we set $B_\gamma = \emptyset$.

\medskip
Our main result is the following theorem.
\begin{theorem}\label{thm1}
	Let $\Omega$ satisfy the assumptions \ref{assumption0} and \ref{assumption}. 
	Let $0 \leq \alpha < \beta$ be such that 
    $B_\alpha \subset \Omega_{\textnormal{in}}$ and $|\Omega| = |B_\beta \setminus \overline{B}_\alpha|$. 
    Then the following assertions hold:
    \begin{enumerate}[label={\rm(\roman*)}]
        \item If $\Omega$ is either symmetric of order $2$ or centrally symmetric, then
	\begin{equation}\label{eq:SWNN}
		\mu_{2}(\Omega) \leq \mu_2(B_\beta \setminus \overline{B}_\alpha).
	\end{equation}
	    \item If $\Omega$ is symmetric of order $4$, then
	\begin{equation}\label{eq:SWNNx}
		\mu_{i}(\Omega) \leq \mu_i(B_\beta \setminus \overline{B}_\alpha)=\mu_2(B_\beta \setminus \overline{B}_\alpha) \quad \text{for}~  i=3,\ldots, N+1,
	\end{equation}
	and
	\begin{equation}\label{eq:SWNNy}
		\mu_{N+2}(\Omega) \leq \mu_{N+2}(B_\beta \setminus \overline{B}_\alpha).
	\end{equation}
	    \item If $N=2$ and $\Omega$ is symmetric of order $8$, then
	\begin{equation}\label{eq:SWNNz}
	\mu_{5}(\Omega) \leq \mu_{4}(B_\beta \setminus \overline{B}_\alpha) =
	\mu_{5}(B_\beta \setminus \overline{B}_\alpha).
	\end{equation}
	\end{enumerate}
	
	If equality holds in \eqref{eq:SWNN}, \eqref{eq:SWNNx}, \eqref{eq:SWNNy}, \eqref{eq:SWNNz}, then $\Omega$ coincides a.e.\ with $B_\beta \setminus \overline{B}_\alpha$.
\end{theorem}

As a direct corollary of Theorem \ref{thm1}, we get the domain monotonicity of several higher Neumann eigenvalues on the class of equimeasurable spherical shells.
\begin{corollary}\label{cor:monotonicity}
    Let $0<\alpha_1 < \alpha$, $0<\beta_1<\beta$, and a ball $B$ be such that $|B_{\beta_1} \setminus \overline{B}_{\alpha_1}| = |B_\beta \setminus \overline{B}_\alpha| = |B|$. Then 
    \begin{equation*}\label{eq:muimui}
	\mu_i(B_\beta \setminus \overline{B}_\alpha) < \mu_i(B_{\beta_1} \setminus \overline{B}_{\alpha_1}) < \mu_i(B)
	\quad \text{for}~ i = 2,3,\dots,N+2,
	\end{equation*}
	and, in the case $N=2$, also
	$$
	\mu_5(B_\beta \setminus \overline{B}_\alpha) < \mu_5(B_{\beta_1} \setminus \overline{B}_{\alpha_1}) < \mu_5(B).
	$$
\end{corollary}

\begin{remark}
    Corollary \ref{cor:monotonicity} shows that the inequalities given by Theorem \ref{thm1} provide the best upper bounds with respect to $\alpha$ if $B_\alpha$ is chosen to be the \textit{maximal} ball (centred at the origin) contained in $\Omega_{\text{in}}$.
    Moreover, the thinner the domain $\Omega$ (i.e., the closer $\alpha$ to $\beta$), the better these upper bounds compared to the estimates by $\mu_i(B)$.
\end{remark}  
    
\begin{remark}  
    If $\alpha>0$, then  \eqref{eq:SWNN} improves \eqref{eq:SW} for the class of symmetric domains described by the assumption \ref{assumption}. 
    (Although in the case $\alpha=0$, \eqref{eq:SWNN} is reduced to the classical Szeg\H{o}-Weinberger inequality \eqref{eq:SW} which holds regardless any symmetry assumptions on $\Omega$.)
	The inequalities \eqref{eq:SWNNx} and \eqref{eq:SWNNy} provide an improvement and a higher-dimensional generalisation of the inequalities \eqref{eq:SWsym} and \eqref{eq:SWsym4}, respectively.
	Moreover, \eqref{eq:SWNNx} improves \cite[Theorem 4.2]{AB}.
	The inequality for $i=3$ in \eqref{eq:SWNNx} also improves  \eqref{eq:GNP}.
  To the best of our knowledge, the inequality \eqref{eq:SWNNz} has not been explored in the literature before. 
\end{remark}

\begin{remark}\label{rem:nonexistence-symmetry}
	The following generalization of the inequality \eqref{eq:SWNNz} to higher dimensions ($N \geq 3$):
	\begin{equation}\label{eq:SWNNz1}
	\mu_{i}(\Omega) \leq \mu_{i}(B_\beta \setminus \overline{B}_\alpha) \quad \text{for} \quad  i=N+3, \dots, \frac{N(N+3)}{2},
	\end{equation}
	can be established for the class of domains which are symmetric of order $8$ with respect to \textit{some} coordinate plane and, simultaneously, symmetric of order $4$ with respect to all other coordinate planes.
	However, we show in Section \ref{sec:existence} that the class of such domains consists only of radially symmetric domains, i.e., of balls and spherical shells. 
	In fact, we show that the domains that are symmetric of order $q$ with  $q \neq 1,2,4$ must be radially symmetric.
	Thus,  the only nontrivial consequence of \eqref{eq:SWNNz1} would be the domain monotonicity of the corresponding higher Neumann eigenvalues in spherical shells as in Corollary \ref{cor:monotonicity}. 
	However, we anticipate that such monotonicity can be obtained by some other, perhaps easier, way. Because of that, we do not provide a proof of \eqref{eq:SWNNz1}.
\end{remark}

For proving Theorem \ref{thm1}, we adapt the original idea of \textsc{Weinberger} \cite{weinberger} to our settings. 
Namely, using certain eigenfunctions of \eqref{eq:N} on $B_\be\setminus \overline{B}_\al$, we  construct trial 
finite-dimensional subspaces of $H^1(\Omega)$
for the variational characterisation of $\mu_i(\Om)$ given by the Courant-Fischer minimax principle. 
These trial subspaces have the property that the maximum of the Rayleigh quotient over them does not exceed $\mu_i(B_\be\setminus \overline{B}_\al)$. 
In \cite{weinberger}, \textsc{Weinberger} produced such construction for an arbitrary domain $\Omega$ using an orthogonal basis of eigenfunctions corresponding to $\mu_2(B)$ in combination with their certain monotonicity properties. 
However, since we are considering higher eigenvalues $\mu_i(\Om)$ with $i\ge 2$ and allow the presence of the ``hole'' $\Omega_{\text{in}}$ in $\Omega$, it is  difficult to guarantee a
similar construction of trial subspaces without  additional assumptions on $\Omega$. 
Counterexamples which we provide in Section \ref{sec:counterexample} demonstrate that symmetry requirements of Theorem \ref{thm1} might be vital.
Moreover, in general, \textsc{Weinberger}'s  argument for the monotonicity works only for the second eigenfunctions, see Remark \ref{rem:weinberger}. 
In order to deal with higher eigenvalues, we provide a more universal argument which also covers the case of domains with ''holes'', see Proposition \ref{prop:bound}.
In fact, the assumption that $B_\alpha$ must be contained in $\Omega_{\text{in}}$ appears only in this proposition.

The structure of this work is as follows. 
Section \ref{sec:spectrum} contains some preliminaries on the structure and properties of the spectra $\{\mu_k(B_\beta \setminus \overline{B}_\alpha)\}$, several results being proved in Appendix \ref{sec:appendix2}.
Section \ref{sec:proof1} is devoted to the proof of Theorem \ref{thm1}.
In Section \ref{sec:counterexample}, we discuss the violation of the obtained inequalities for domains outside of the considered symmetry classes.
In Section \ref{sec:existence}, we discuss the existence and properties of nonradial domains satisfying the symmetry assumptions imposed in Theorem \ref{thm1} and Remark \ref{rem:nonexistence-symmetry}.
Section \ref{sec:final} contains some concluding remarks.
Finally, Appendix \ref{sec:appendix} contains several auxiliary results needed for the proof of the main theorem.

\section{Spectrum of \texorpdfstring{\eqref{eq:N}}{(EP)} on radially symmetric domains}\label{sec:spectrum}
In this section, we provide several results on the structure and properties of eigenvalues and eigenfunctions of the problem \eqref{eq:N} in the spherical shell $B_\beta \setminus \overline{B}_\alpha$, where $0 \leq \alpha<\beta<\infty$. 
Recall that in the case $\alpha=0$, we set $B_\beta \setminus \overline{B}_\alpha \equiv B_\beta$.
Hereinafter, we will denote $\mathbb{N}_0 = \{0,1,2,\dots\}$ and $\mathbb{N} = \{1,2,\dots\}$. 

In the spherical coordinates $(r,\omega) \in (0,+\infty) \times S^{N-1}$, the Laplacian acts on a smooth function $u=u(x)=u(r,\omega)$ as
\begin{equation*}\label{Spherical}
    \Delta u = \frac{\partial^2 u}{\partial r^2} + \frac{N-1}{r} \frac{\partial u}{\partial r} + \frac{1}{r^2} \Delta_{S^{N-1}} u,
\end{equation*}
where $r = |x|$ and $\Delta_{S^{N-1}}$ is the Laplace-Beltrami operator on the unit sphere $S^{N-1}$.
For a smooth function $h$ on $S^{N-1}$, the action of $\Delta_{S^{N-1}}$ is given as
$$
\Delta_{S^{N-1}}h
=
\left.\Delta h\left(\frac{x}{|x|}\right)\right|_{|x|=1}.
$$
It is well-known that the set of all eigenfunctions of $\Delta_{S^{N-1}}$ is precisely the set of spherical harmonics, which are defined as the restriction to $S^{N-1}$ of homogeneous harmonic polynomials in $N$ variables. 
Denote by $H_l$ the set of all homogeneous harmonic polynomials in $N$ variables and of degree $l \in \mathbb{N}_0$.
Clearly, $H_0$ consists only of constant functions.
In Appendix \ref{sec:Orthogonal}, we discuss the form of orthogonal bases of $H_1$ and $H_2$, which will be important in the proof of Theorem \ref{thm1}.
We will need the following result on the spectrum of $\Delta_{S^{N-1}}$, see, e.g., \cite[Sections 22.3, 22.4]{shubin}.

\begin{proposition}\label{prop:multiplicity}
The spectrum of $\Delta_{S^{N-1}}$ is the set $\left\{-l(l+N-2):\, l\in \mathbb{N}_0 \right\}$. 
The multiplicity of the eigenvalue $-l(l+N-2)$ is equal to the dimension of $H_l$ and it is given by
\begin{equation*}\label{multiplicity}
   \Lambda_l 
   =
   \mathrm{dim} H_l
   =
   \binom{l+N-1}{N-1} - \binom{l+N-3}{N-1}
   =
   \frac{2l + N-2}{l+N-2}
   \binom{l+N-2}{l}.
\end{equation*}
\end{proposition}

By separating the variables,
one can find a complete orthogonal system (in fact, a
basis) in $L^2(B_\beta \setminus \overline{B}_\alpha)$ of eigenfunctions of \eqref{eq:N} on $B_\beta \setminus \overline{B}_\alpha$ in the form
\begin{equation*}\label{eq:eigenf}
\varphi(x) 
=
v(|x|) h\left(\frac{x}{|x|}\right), 
\quad x \in \left(B_\beta \setminus \overline{B}_\alpha\right) \setminus \{0\},
\end{equation*}
cf.\ \cite[Chapter II, \S 1.6]{courant}.
Here, $h$ is a spherical harmonic
corresponding to the eigenvalue $-l(l+N-2)$ of $\Delta_{S^{N-1}}$, and $v$ is an eigenfunction of the Sturm-Liouville eigenvalue problem (\textit{SL problem}, for short)
\begin{equation}\label{eq:ode}
-v'' - \frac{N-1}{r} v' + \frac{l(l+N-2)}{r^2} v = \mu v, 
\quad r \in (\alpha,\beta),
\end{equation}
with the boundary conditions
\begin{equation}\label{eq:bcN}
v'(\alpha)=0 \quad \text{and} \quad v'(\beta)=0.
\end{equation}
By the standard Sturm-Liouville theory, for every $l \in \mathbb{N}_0$ the spectrum of the SL problem \eqref{eq:ode}, \eqref{eq:bcN} consists of a sequence of eigenvalues
\begin{equation}\label{Neumann evs}
 (0 \le )~ \mu_{l,1} < \mu_{l,2} < \dots < \mu_{l,k}\to +\infty \quad \text{as}~ k \to +\infty. 
\end{equation}
Each eigenvalue $\mu_{l,j}$ is simple and the associated eigenfunction vanishes exactly $j-1$ times in $(\al,\be)$. 
In particular, any first eigenfunction has a constant sign in $(\al,\be)$. 
Moreover, $\mu_{0,1}=0$ and the associated eigenfunction is a nonzero constant.

The spectrum of the problem \eqref{eq:N} on $B_\beta \setminus \overline{B}_\alpha$ is given by 
\begin{equation}\label{eq:charmu}
\{\mu_k(B_\beta \setminus \overline{B}_\alpha)\}_{k \in \mathbb{N}} = 
\{\mu_{l,j}\}_{l \in \mathbb{N}_0, j \in \mathbb{N}},
\end{equation}
where each $\mu_{l,j}$ is counted with multiplicity $\Lambda_l$ (the dimension of $H_l$).
In particular, if $\mu_k(B_\beta \setminus \overline{B}_\alpha) = \mu_{l,j}$ for some $k,l,j$, then the multiplicity of $\mu_k(B_\beta \setminus \overline{B}_\alpha)$ is at least $\Lambda_l$.
If $\alpha=0$, that is, in the case of the ball, then the multiplicity of such $\mu_k(B_\beta \setminus \overline{B}_\alpha)$ is exactly $\Lambda_l$, see \cite[Proposition 2.3]{HS}.
However, if $\alpha>0$, then it might happen that $\mu_{l_1,j_1} = \mu_{l_2,j_2}$ for two different pairs of indices.
In this case, the multiplicity of the corresponding eigenvalue of \eqref{eq:N} is at least $\Lambda_{l_1} + \Lambda_{l_2}$.
Notice also that taking $l=0$ we obtain all the radial eigenvalues of \eqref{eq:N}.

Since the equation \eqref{eq:ode} can be rewritten as
\begin{equation}\label{eq:ode3}
-(r^{N-1}v')' + l(l+N-2) r^{N-3} v = \mu r^{N-1} v, 
\quad r \in (\alpha,\beta),
\end{equation}
it is possible to characterise any eigenvalue $\mu_{l,j}$ as a critical value of the Rayleigh quotient
\begin{equation}\label{eq:RQ}
R_l(v)=\frac{\int_\al^\be \left[(v'(r))^2+\frac{l(l+N-2)}{r^2}v^2(r)\right]r^{N-1}\,dr}{\int_\al^\be v^2(r)r^{N-1}\,dr}, 
\quad 
v\in H^1((\al,\be);r^{N-1})\setminus\{0\},
\end{equation}
where $H^1((\al,\be);r^{N-1})$ is the weighted Sobolev space on $(\alpha,\beta)$ with the weight $r^{N-1}$. 
More precisely, using the Courant-Fischer minimax formula, we have
\begin{equation}\label{eq:mudef}
\mu_{l,j}=\min_{X\in \mathcal{X}_j}\max_{u\in X\setminus \{0\}}R_l(u),
\end{equation}
where $\mathcal{X}_j$ is the collection of all $j$-dimensional subspaces of $H^1((\al,\be);r^{N-1})$.
In particular,
\begin{equation}\label{eq:mudef1}
\mu_{l,1}=\min_{u \in H^1((\al,\be);r^{N-1}) \setminus \{0\}} R_l(u).
\end{equation}
Since $R_l(v)$ is strictly increasing with respect to $l$, we deduce from \eqref{eq:mudef} that for each fixed $j\in \mathbb{N}$, 
\begin{equation}\label{eqn:col}
    \mu_{0,j}
    < 
    \mu_{1,j}
    < 
    \ldots 
    < 
    \mu_{l,j} 
    < \ldots
\end{equation}

Note that the general solution of the equation \eqref{eq:ode} is given by
\begin{equation*}\label{eq:generalsolution}
v(r) 
= r^{1-\frac{N}{2}} 
\left[ 
c_1 J_{\frac{N}{2}+l-1}\left(\sqrt{\mu} r\right) 
+ c_2Y_{\frac{N}{2}+l-1}\left(\sqrt{\mu} r\right)
\right],
\end{equation*}
where $J_s$ and $Y_s$ are the  Bessel functions of order $s$, of the first and second kind, respectively. 
In the case $\alpha=0$, $c_2=0$ in view of the singularity of $Y_{\frac{N}{2}+l-1}$ at zero, and one can characterize each eigenvalue $\mu_{l,j}$ as $\mu_{l,j} = \beta^{-2} \left(p_{\frac{N}{2},j}^{(l)}\right)^2$ (see Section \ref{sec:introduction} for the definition of $p_{\frac{N}{2},j}^{(l)}$).
In the case $\alpha>0$, the constants $c_1$ and $c_2$ are determined through the boundary conditions \eqref{eq:bcN}, and $\mu_{l,j}$ can be characterised as 
the $j$-th positive zero of the following cross-product of Bessel functions:
\begin{align*}
 F(\mu) 
 := 
 &\left(\frac{2-N}{2} J_{\nu}\left(\sqrt{\mu} \alpha\right)
 + \alpha \sqrt{\mu} J'_{\nu}\left(\sqrt{\mu} \alpha\right) \right)
 \left(\frac{2-N}{2} Y_{\nu}\left(\sqrt{\mu} \beta\right)
 + \beta \sqrt{\mu} Y'_{\nu}\left(\sqrt{\mu} \beta\right) \right)
 \\
 &-
 \left(\frac{2-N}{2} Y_{\nu}\left(\sqrt{\mu} \alpha\right)
 + \alpha \sqrt{\mu} Y'_{\nu}\left(\sqrt{\mu} \alpha\right) \right)
 \left(\frac{2-N}{2} J_{\nu}\left(\sqrt{\mu} \beta\right)
 + \beta \sqrt{\mu} J'_{\nu}\left(\sqrt{\mu} \beta\right) \right),
\end{align*}
where $\nu = \frac{N}{2}+l-1$.

\begin{remark}
In the case $\alpha>0$, we have 
\begin{equation}\label{eq:muk-conv}
 \mu_k(B_\beta \setminus \overline{B}_{\alpha}) 
 \to 
 \mu_k(B_\beta)
 \quad \text{as}~ \alpha \to 0,
 \end{equation}
 for any $k \in \mathbb{N}_0$, see e.g., \cite[Theorem 3.5 and Corollary 3.6]{daners}. 
 In particular, if we temporarily denote $\mu_{l,j} = \mu_{l,j}(\alpha)$ to stress the dependence on $\alpha$, then the convergence \eqref{eq:muk-conv} together with the characterisation \eqref{eq:charmu} 
 and the fact that $\mu_{l_1,j_1}(0) \neq \mu_{l_2,j_2}(0)$ provided $(l_1,j_1) \neq (l_2,j_2)$ (see \cite[Lemma 2.5]{HS})
 yield
\begin{equation}\label{eq:mulj-conv}
 \mu_{l,j}(\alpha) 
 \to 
 \mu_{l,j}(0)
 \quad \text{as}~ \alpha \to 0,
 \end{equation}
 for any $l \in \mathbb{N}_0$ and $j \in \mathbb{N}$.
\end{remark}
 
From \eqref{Neumann evs} and \eqref{eqn:col} we see  that the entries of the infinite matrix $\{\mu_{l,j}\}$ are increasing along   the rows and columns. 
Since the eigenvalues of \eqref{eq:N} are counted in the nondecreasing order, the first and second eigenvalues of \eqref{eq:N} on $B_\beta \setminus \overline{B}_\alpha$ must be
$$
\mu_1(B_\beta \setminus \overline{B}_\alpha)=\mu_{0,1} = 0
\quad \text{and} \quad 
\mu_2(B_\beta \setminus \overline{B}_\alpha)=\min\{\mu_{1,1},\mu_{0,2}\}.
$$ 
In the following lemma, we provide a precise ordering of the eigenvalues $\mu_{1,1}$, $\mu_{2,1}$, and $\mu_{0,2}$, which gives, in particular, that $\mu_2(B_\beta \setminus \overline{B}_\alpha) = \mu_{1,1}$.
\begin{lemma}\label{lem:n+2}
We have 
\begin{equation}\label{eq:mu11<mu21<mu02}
\mu_{1,1}<\mu_{2,1}<\mu_{0,2}.
\end{equation}
\end{lemma}
The proof of Lemma \ref{lem:n+2} is placed in Appendix \ref{sec:appendix2}.
Notice that the weaker inequality $\mu_{1,1}<\mu_{0,2}$ can be obtained from \cite[Theorems 1.4]{AK} or \cite[Theorem 1.2]{Li}, see also Proposition \ref{prop:mu2tau2} for a simple proof. However,  \eqref{eq:mu11<mu21<mu02} cannot be improved, in general, to $\mu_{3,1} < \mu_{0,2}$. 
Indeed, in the planar case $N=2$ with $\alpha=0$ and $\beta=1$, one has $\mu_{3,1} \approx 17.65$, while $\mu_{0,2} \approx 14.68$. 
Hence $\mu_{3,1} > \mu_{0,2}$ holds for all sufficiently small $\alpha \geq 0$ in view of the convergence \eqref{eq:mulj-conv}.

Thanks to Proposition \ref{prop:multiplicity}, the following corollary of Lemma \ref{lem:n+2} can be easily derived.
\begin{corollary}\label{cor:multiplicity}
For any $0 \leq \al < \be,$ we have
\begin{align}
\notag
\mu_2(B_\beta \setminus \overline{B}_\alpha)
&=
\dots
=
\mu_{N+1}(B_\beta \setminus \overline{B}_\alpha)
= 
\mu_{1,1},\\
\label{eq:multip2}
\mu_{N+2}(B_\beta \setminus \overline{B}_\alpha)
&=
\dots
=
\mu_{\frac{N(N+3)}{2}}(B_\beta \setminus \overline{B}_\alpha)  =
\mu_{2,1}.
\end{align}
\end{corollary}

\begin{remark}\label{rem:radial}
By Corollary \ref{cor:multiplicity} and the second inequality in \eqref{eq:mu11<mu21<mu02}, any eigenfunction $\varphi_k$ corresponding to the eigenvalue $\mu_k(B_\beta \setminus \overline{B}_\alpha)$ with $k=2,\ldots,N+1$
has the form
\begin{equation*}\label{eq:eigenf2}
\varphi_k(x) 
=
v(|x|) h\left(\frac{x}{|x|}\right)
=
\frac{v(r)}{r} h(x), 
\quad x \in \left(B_\beta \setminus \overline{B}_\alpha\right) \setminus \{0\},
\end{equation*}
where $h \in H_1$, and if $k=N+2,\ldots, \frac{N(N+3)}{2}$, then $\varphi_k$ has the form
\begin{equation*}\label{eq:eigenf22}
	\varphi_k(x) 
	=
	v(|x|) h\left(\frac{x}{|x|}\right)
	=
	\frac{v(r)}{r^2} h(x), 
	\quad x \in \left(B_\beta \setminus \overline{B}_\alpha\right) \setminus \{0\},
\end{equation*}
where $h \in H_2$.
In particular, any $\varphi_k$ with $k=2,\ldots, \frac{N(N+3)}{2}$
is nonradial, and it is an odd function when $k=2,\ldots, N+1$, i.e.,  
$$
\varphi_k(-x) = -\varphi_k(x)
\quad
\text{for any}~
x \in B_\beta \setminus \overline{B}_\alpha.
$$
Let us remark that it is not known whether any second eigenfunction of the problem \eqref{eq:N} in a general centrally symmetric domain $\Omega$ which is homeomorphic to a spherical shell is odd, see \cite{kennedy}. 
\end{remark}

\begin{remark}\label{rem:multip}
	The highest index in \eqref{eq:multip2}  occurs as
	$$
	\frac{N(N+3)}{2} = N+1 + \frac{(N+2)(N-1)}{2}, 
	$$
	where $\frac{(N+2)(N-1)}{2} = \Lambda_2$ is the dimension of $H_2$, see Proposition \ref{prop:multiplicity}.
\end{remark}

The following auxiliary lemma will be needed to obtain Proposition  \ref{prop:bound} below, see Appendix \ref{sec:appendix2} for the proof.
\begin{lemma}\label{lem:v}
	Let $l \in \mathbb{N}$ and let
	$v$ be a positive  eigenfunction corresponding to the eigenvalue $\mu_{l,1}$ of the SL problem \eqref{eq:ode}, \eqref{eq:bcN}.
	Then for any $r \in (\alpha,\beta)$ we have  $v'(r) > 0$  and
	\begin{equation}\label{eq:long}
    \left(\frac{l(l+N-2)}{r^2}
    -
    \mu_{l,1}\right)v^2(r)
    \geq
    \left(\frac{l(l+N-2)}{\beta^2}
    -
    \mu_{l,1}\right)v^2(\beta).
	\end{equation}
\end{lemma}

Finally, we establish the following general result which will be important in the proof of the main theorem.
\begin{proposition}\label{prop:bound}
		Let $\Omega$ be a bounded domain satisfying the assumption \ref{assumption}. 
		Let $0 \leq \alpha < \beta$ be such that 
		$B_\alpha \subset \Omega_{\textnormal{in}}$ and $|\Omega| = |B_\beta \setminus \overline{B}_\alpha|$. 
		Let $l \in \mathbb{N}$ and let $v$ be a positive  eigenfunction corresponding to the eigenvalue $\mu_{l,1}$ of the SL problem \eqref{eq:ode}, \eqref{eq:bcN}.
		Define 
		\begin{equation}\label{eq:G}
        G_l(r)
        =
        \left\{
        \begin{aligned}
      &v(r) 	&&\text{if}~ r \in (\alpha,\beta), \\
      &v(\beta) &&\text{if}~ r \geq \beta.
        \end{aligned}
        \right.
        \end{equation}
		Then
	   \begin{equation}\label{eq:bound1}
	   \frac{\int_{\Omega}
	   	\left(
		(G_l'(r))^2 + \frac{l(l+N-2)G_l^2(r)}{r^2}
		\right) 
		dx}
		{\int_{\Omega} G_l^2(r) \, dx}
		\leq \mu_{l,1},
	   \end{equation} 
	   and equality holds in \eqref{eq:bound1} if and only if $\Omega$ coincides a.e.\ with $B_\beta \setminus \overline{B}_\al$.
	\end{proposition}
	\begin{proof}	
	Denote $G(r)=G_l(r)$ and $H(r)=(G'(r))^2 + \frac{l(l+N-2)G^2(r)}{r^2}$, for brevity.
	We see from \eqref{eq:mudef1} that 
	\begin{equation}\label{eq:bound2}
	\mu_{l,1}
	=
	\frac{\int_\al^\be
	H(r)r^{N-1} \,dr}{\int_\al^\be v^2(r)r^{N-1} \,dr}
	=
	\frac{\int_{B_\beta \setminus \overline{B}_\alpha} H(r) \, dx}
	{\int_{B_\beta \setminus \overline{B}_\alpha} G^2(r) \, dx}.
	\end{equation}
	Thus, the desired inequality \eqref{eq:bound1} is equivalent to 
\begin{equation}\label{eq:H<G}
    \frac{\int_{\Om} H(r) \, dx}
{\int_{\Om} G^2(r) \, dx}\leq \frac{\int_{B_\beta \setminus \overline{B}_\alpha} H(r) \, dx}
{\int_{B_\beta \setminus \overline{B}_\alpha} G^2(r) \, dx}
.
\end{equation}
In order to prove \eqref{eq:H<G}, we first represent $\Omega$ as a union of disjoint sets as follows:
\begin{align*}
\Om
= [\Om \cap (B_\beta \setminus \overline{B}_\alpha)] \cup [\Om \cap (B_\beta \setminus \overline{B}_\alpha)^c]
= [\Om \cap (B_\beta \setminus \overline{B}_\alpha)] \cup [\Om \cap B_\beta^c] \cup [\Om\cap \overline{B}_\alpha].
\end{align*}
Similarly,
\begin{align*}
 B_\beta \setminus \overline{B}_\alpha =& [\Om \cap (B_\beta \setminus \overline{B}_\alpha)]\cup [\Om^c \cap (B_\beta \setminus \overline{B}_\alpha)].
\end{align*}
From the choice of  $\al$ and $\be$, we have  $|\Om|=| B_\beta \setminus \overline{B}_\alpha|$ and $|\Om\cap \overline{B}_\alpha|=0$, which yields
\begin{equation}\label{eq:measures}
 |\Om^c \cap (B_\beta \setminus \overline{B}_\alpha)|=|\Om \cap B_\beta^c|.
\end{equation} 
Therefore,
\begin{align}
\notag
\int_{\Omega} H(r) \, dx 
&= \int_{\Omega\cap (B_\beta \setminus \overline{B}_\alpha)}  H(r)\, dx 
+
\int_{\Omega \cap B_\beta^c} H(r) \, dx\\
\label{eq:H1}
&= 
\int_{B_\beta \setminus \overline{B}_\alpha}  H(r)\, dx - \int_{\Om^c \cap (B_\beta \setminus \overline{B}_\alpha)}  H(r)\, dx
+
\int_{\Omega \cap B_\beta^c} H(r) \, dx
\end{align}
and, in the same manner, 
\begin{equation}\label{eq:G1}
\int_{\Omega} G^2(r) \, dx
=  \int_{B_\beta \setminus \overline{B}_\alpha}  G^2(r)\, dx 
- 
\int_{\Om^c \cap (B_\beta \setminus \overline{B}_\alpha)}  G^2(r)\, dx
+
\int_{\Omega \cap B_\beta^c} G^2(r) \, dx.
\end{equation}
Substituting \eqref{eq:H1} and \eqref{eq:G1} into \eqref{eq:H<G} and rearranging, we see that \eqref{eq:H<G} is satisfied if and only if
\begin{align}
\notag
&- \int_{\Om^c \cap (B_\beta \setminus \overline{B}_\alpha)}  H(r)\, dx \int_{B_\beta \setminus \overline{B}_\alpha}  G^2(r)\, dx 
+
\int_{\Omega \cap B_\beta^c} H(r) \, dx \int_{B_\beta \setminus \overline{B}_\alpha}  G^2(r)\, dx \\
\label{eq:H<G1}
&\leq
- \int_{\Om^c \cap (B_\beta \setminus \overline{B}_\alpha)}  G^2(r)\, dx \int_{B_\beta \setminus \overline{B}_\alpha}  H(r)\, dx 
+
\int_{\Omega \cap B_\beta^c} G^2(r) \, dx \int_{B_\beta \setminus \overline{B}_\alpha}  H(r)\, dx.
\end{align}
Dividing both sides of \eqref{eq:H<G1} by $\int_{B_\beta \setminus \overline{B}_\alpha}  G^2(r)\, dx$ and using \eqref{eq:bound2}, 
we see that \eqref{eq:H<G1} is equivalent to 
\begin{equation}\label{eq:H<G2}
\int_{\Omega \cap B_\beta^c} H(r) \, dx - \int_{\Om^c \cap (B_\beta \setminus \overline{B}_\alpha)}  H(r)\, dx
\leq
\mu_{l,1}\left(
\int_{\Omega \cap B_\beta^c} G^2(r) \, dx
-
\int_{\Om^c \cap (B_\beta \setminus \overline{B}_\alpha)}  G^2(r)\, dx
\right).
\end{equation}
Notice now that for any $x \in \Omega \cap B_\beta^c$ there holds $|x| \geq \beta$, and hence
\begin{equation*}\label{eq:eqGH}
G(|x|)=G(\beta)
\quad \text{and} \quad
H(|x|)
=
\frac{l(l+N-2)G^2(\beta)}{|x|^2}
\leq
\frac{l(l+N-2)G^2(\beta)}{\beta^2}=H(\beta),
\end{equation*}
where the inequality for $H$ is strict if $|x| > \beta$. 
This yields, in view of \eqref{eq:measures}, 
\begin{align}
\label{eq:Hstrict}
\int_{\Omega \cap B_\beta^c} H(r) \, dx
&\leq
\int_{\Omega \cap B_\beta^c} H(\beta) \, dx
=
\int_{\Om^c \cap (B_\beta \setminus \overline{B}_\alpha)}
H(\beta) \,dx,\\
\label{eq:Gstrict}
\int_{\Omega \cap B_\beta^c} G^2(r) \, dx
&=
\int_{\Omega \cap B_\beta^c} G^2(\beta) \, dx
=
\int_{\Om^c \cap (B_\beta \setminus \overline{B}_\alpha)} G^2(\beta) \, dx,
\end{align}
where the inequality \eqref{eq:Hstrict} is strict if and only if $|\Omega \cap B_\beta^c| > 0$. 
Thus, using \eqref{eq:Hstrict} and \eqref{eq:Gstrict}, we conclude that \eqref{eq:H<G2} is satisfied provided 
$$
\int_{\Om^c \cap (B_\beta \setminus \overline{B}_\alpha)}
\left[
H(\beta)-H(r)-\mu_{l,1}\left(G^2(\beta)-G^2(r)\right)
\right]
dx\le 0,
$$ 
or, equivalently, 
$$
\int_{\Om^c \cap (B_\beta \setminus \overline{B}_\alpha)}\left[
\left(
\frac{l(l+N-2)}{\beta^2}-
\mu_{l,1}\right)v^2(\beta)
-\left(
 \frac{l(l+N-2)}{r^2}-
\mu_{l,1}\right)v^2(r)
- (v'(r))^2
\right] 
dx \leq 0.
$$
Lemma \ref{lem:v} asserts that the  above integrand is negative on $(\alpha,\beta)$, which
completes the proof of the inequality \eqref{eq:bound1} and shows that if $|\Omega \cap B_\beta^c| > 0$ or, equivalently, $|\Om^c \cap (B_\beta \setminus \overline{B}_\alpha)|>0$, then \eqref{eq:bound1} is strict.
Clearly, if $|\Omega \cap B_\beta^c| = 0$ or, equivalently, $|\Om^c \cap (B_\beta \setminus \overline{B}_\alpha)| = 0$, then 
$\Omega \subset B_\beta \setminus 
\overline{B}_\alpha$, and hence  $|(B_\beta \setminus 
\overline{B}_\alpha) \setminus \Omega | = 0$.
That is, equality holds in \eqref{eq:bound1} if and only if $\Omega$ coincides a.e.\ with $B_\beta \setminus 
\overline{B}_\alpha$.
\end{proof}

\begin{remark}\label{rem:weinberger}
In the original proof of \textsc{Weinberger} \cite{weinberger}, the inequality 
\eqref{eq:bound1} (or, equivalently,
\eqref{eq:H<G}) for $l=1$ and $\alpha=0$ was proved  by showing that $H(r)$ decreases and $G(r)$ increases on $(0,\beta)$, and hence
\begin{equation}\label{eq:hhgg}
\int_{\Omega} H(r) \, dx 
\leq \int_{B_\beta} H(r)\, dx 
\quad \text{and} \quad 
\int_{\Omega} G^2(r) \, dx 
\geq \int_{B_\beta}  G^2(r)\, dx.
\end{equation}
For $l=1$ and $\alpha>0$, the above inequalities are also satisfied. 
However, according to our numerical simulation, the first inequality in \eqref{eq:hhgg} does not hold, in general, for $l\ge 2$, since $H(r)$ might not be a decreasing function.  
Thus, our argument presented in the proof of Proposition \ref{prop:bound} is more universal.
\end{remark}

\section{Proof of Theorem \ref{thm1}}\label{sec:proof1}

Let $\mathcal{X}_k$ be the collection of all $k$-dimensional subspaces of $H^1(\Om)$ that are orthogonal to the (one-dimensional) subspace of constant functions. Then, for any $k \in \mathbb{N}$,
by the Courant-Fischer minimax formula,  
\begin{equation}\label{eq:mukdef}
\mu_{k+1}(\Omega) = \min_{X\in\mathcal{X}_k} \max_{u \in X\setminus \{0\}} \frac{\int_{\Omega} |\nabla u|^2 \, dx}{\int_{\Omega} u^2 \, dx}.
\end{equation}
In particular, the second eigenvalue of the problem \eqref{eq:N} is defined as
\begin{equation}\label{eq:mu2def}
\mu_2(\Omega) = \min\left\{\frac{\int_{\Omega} |\nabla u|^2 \, dx}{\int_{\Omega} u^2 \, dx}: u \in H^1(\Omega) \setminus \{0\},~ \int_{\Omega} u \, dx = 0\right\}.
\end{equation}

We divide the proof of Theorem \ref{thm1} into four subsections according to the consideration of the inequalities \eqref{eq:SWNN}, \eqref{eq:SWNNx}, \eqref{eq:SWNNy}, and \eqref{eq:SWNNz}.

\subsection{Proof of \eqref{eq:SWNN}}\label{sec:proof11}
Recalling that $\mu_2(B_\be\setminus \overline{B}_\al)=\mu_{1,1}$ by  Corollary \ref{cor:multiplicity}, we are going to show that
\begin{equation*}\label{eq:SWNNfirst}
\mu_{2}(\Omega) 
\leq \mu_{1,1}.
\end{equation*} 
To provide appropriate trial functions for the variational characterization \eqref{eq:mu2def} of $\mu_{2}(\Omega)$, let us consider the function $G_1(r)$ defined in Proposition \ref{prop:bound} with $l=1$. 
Since $v'(\beta)=0$, we see that $G_1(r)$ is at least a $C^1$-function on $(\alpha,+\infty)$.
For each $i \in \{1,2, \ldots, N\}$, consider the function
$\frac{G_1(r)}{r}x_i$, where $r=|x|$ and $x = (x_1,x_2,\dots,x_N)\in \Omega$.
Notice that $\frac{G_1(r)}{r}x_i\in H^1(\Omega)$. 
By
Proposition \ref{2Orthogonality00} \ref{2Orthogonality00:1} (if $N = 2$) or
Proposition \ref{Orthogonality00} (if $N \geq 3$) we have
\begin{equation}\label{eq:orthog1}
\int_{\Omega} \frac{G_1(r)}{r}x_i \, dx = 0, \quad i=1,\dots,N,
\end{equation}
in view of the symmetry of order $2$ or central symmetry of $\Omega$.
Thus, each $\frac{G_1(r)}{r}x_i$ is a valid trial function for \eqref{eq:mu2def}, and hence
\begin{equation}\label{eq:test}
\mu_2(\Om) \int_{\Omega} \frac{G_1(r)^2}{r^2}x_i^2 \, dx\leq \int_{\Omega} \left|\nabla \left(\frac{G_1(r)}{r}x_i\right)\right|^2  dx 
\end{equation}
for all $i \in \{1,\dots,N\}$.
Moreover, from Remark \ref{remark:norm} we have
$$
\int_{\Omega} \left|\nabla \left(\frac{G_1(r)}{r}x_i\right)\right|^2 dx
=
\int_{\Omega}
\left(
\frac{(G_1'(r))^2}{r^2}x_i^2 - \frac{G_1^2(r)}{r^4}x_i^2 + \frac{G_1^2(r)}{r^2}
\right)
dx.
$$
Summing over $i$, we derive from \eqref{eq:test} that
\begin{equation}\label{eq:taub}
\mu_2(\Omega)
\leq 
\frac{\int_{\Omega}
	\left(
	(G_1'(r))^2 + \frac{(N-1)G_1^2(r)}{r^2}
	\right)
	dx}
{\int_{\Omega} G_1^2(r) \, dx}\leq \mu_{1,1},
\end{equation}
where the last inequality is given by Proposition \ref{prop:bound} with $l=1$. 
This establishes the inequality \eqref{eq:SWNN}.
Moreover, if equality holds in \eqref{eq:SWNN}, then it follows from \eqref{eq:taub} and Proposition \ref{prop:bound} that $\Omega$ coincides a.e.\ with $B_\beta \setminus \overline{B}_\al$.
	
\subsection{Proof of \eqref{eq:SWNNx}}
In view of Corollary \ref{cor:multiplicity}, to establish the inequality \eqref{eq:SWNNx} it is enough to prove that 
\begin{equation*}\label{SW3}
\mu_{N+1}(\Omega) \leq \mu_{1,1}
\end{equation*}
under the assumption that $\Omega$ is symmetric of order $4$.
As an admissible choice of the $N$-dimensional subspace of $H^1(\Omega)$ for the variational characterization \eqref{eq:mukdef} of $\mu_{N+1}(\Omega)$, we take
$$
X_N = \text{span}\left\{ \frac{G_1(r)}{r}x_i,~ i=1,2,\ldots, N \right\},
$$
where $G_1(r)$ is defined in Proposition \ref{prop:bound} with $l=1$.
From \eqref{eq:orthog1}, $\int_\Omega u \, dx =0$ for any $u\in X_N$, and hence we indeed have $X_N \in \mathcal{X}_N$. 
Moreover, by the symmetry of order $4$, we deduce from Lemma \ref{lem:z1z2z3-tildeN2} and Remark \ref{rem:lem:z1} (if $N=2$) or Lemma \ref{lem:z1z2z3-tilde} (if $N \geq 3$) that
\begin{equation}\label{eq:ggg'g'}
\int_\Omega \left(\frac{G_1(r)}{r}x_i\right)  \left(\frac{G_1(r)}{r}x_j\right) \,dx=0 
\quad \text{and} \quad
\int_\Omega \nabla \left(\frac{G_1(r)}{r}x_i\right)  \nabla \left(\frac{G_1(r)}{r}x_j\right) dx =0
\end{equation}
for any $i \neq j$.
At the same time, by Proposition \ref{2Orthogonality00} \ref{2Orthogonality00:2} (if $N=2$) or Proposition \ref{Orthogonality0} \ref{Orthogonality0:3} (if $N \geq 3$), and Remark \ref{remark:norm}, there exist constants $A_1,A_2>0$
such that
\begin{align*}
\int_\Omega \left(\frac{G_1(r)}{r}x_i\right)^2 \, dx
&=
\int_\Omega \frac{G_1^2(r)}{r^2}x_i^2\,dx
=
A_1,\\
\int_\Omega \left|\nabla\left(\frac{G_1(r)}{r} x_i\right)\right|^2dx 
&=\int_{\Omega}
\left(
\frac{(G_1'(r))^2}{r^2}x_i^2 - \frac{G_1^2(r)}{r^4}x_i^2 + \frac{G_1^2(r)}{r^2}
\right)
dx
= 
A_2,
\end{align*}
for any $i \in \{1,2,\ldots, N\}$.
Therefore,
\begin{align*}
NA_1
&=
\sum_{i=1}^N\int_\Omega \left(\frac{G_1(r)}{r}x_i\right)^2  dx
= 
\int_\Omega G_1^2(r) \,dx,\\
NA_2
&=
\sum_{i=1}^N
\int_{\Omega}
\left(
\frac{(G_1'(r))^2}{r^2}x_i^2 - \frac{G_1^2(r)}{r^4}x_i^2 + \frac{G_1^2(r)}{r^2}
\right)
dx
=
\int_{\Omega}
\left((G_1'(r))^2+ \frac{(N-1)G_1^2(r)}{r^2}\right)dx.
\end{align*}
Thus, for each  $i \in \{1,2,\ldots, N\}$ we have
\begin{align}
\int_\Omega \left(\frac{G_1(r)}{r}x_i\right)^2 dx
&=A_1
=
\frac{1}{N}\int_\Omega G_1^2(r) \,dx,\label{eq:A}\\
\int_\Omega \left|\nabla\left(\frac{G_1(r)}{r} x_i\right)\right|^2dx 
&=A_2
=
\frac{1}{N}\int_{\Omega}
\left((G_1'(r))^2+ \frac{(N-1)G_1^2(r)}{r^2}\right)dx.\label{eq:A'}
\end{align}

Since for any $u\in X_N \setminus \{0\}$ there exist $c_1,c_2, \ldots, c_N \in \mathbb{R}$, not simultaneously equal to zero, such that
$$
u = c_1 \frac{G_1(r)}{r}x_1 + \dots + c_N \frac{G_1(r)}{r}x_N,
$$
the orthogonality \eqref{eq:ggg'g'} and the expressions \eqref{eq:A} and \eqref{eq:A'} imply that
\begin{align*} 
\frac{\int_{\Omega} |\nabla u|^2 \, dx}{\int_{\Omega} u^2 \, dx}
&=
\frac{\sum_{i=1}^N c_i^2\int_{\Omega} \left|\nabla \left(\frac{G_1(r)}{r}x_i\right)\right|^2 dx}
{\sum_{i=1}^N c_i^2\int_{\Omega} \left(\frac{G_1(r)}{r}x_i\right)^2 dx}
=
\frac{A_2}{A_1}
=
\frac{\int_{\Omega}
\left((G_1'(r))^2+ \frac{(N-1)G_1^2(r)}{r^2}\right) dx}
{\int_\Omega G_1^2(r) \,dx }.
\end{align*}
Therefore, for any  $u\in X_N \setminus \{0\}$, by Proposition \ref{prop:bound} with $l=1$ we get
$$
\frac{\int_{\Omega} |\nabla u|^2 \, dx}{\int_{\Omega} u^2 \, dx}\le \mu_{1,1},
$$ 
and equality holds if and only if $\Omega$ coincides a.e.\ with $B_\beta \setminus \overline{B}_\al$.
Finally, by the Courant-Fischer minimax formula \eqref{eq:mukdef}, 
\begin{equation*}\label{eq:muN1}
\mu_{N+1}(\Omega) \leq \max_{u \in X_N\setminus \{0\}} \frac{\int_{\Omega} |\nabla u|^2 \, dx}{\int_{\Omega} u^2 \, dx} \le \mu_{1,1}, 
\end{equation*}
which completes the proof of \eqref{eq:SWNNx}.

\subsection{Proof of \eqref{eq:SWNNy}}\label{section:SWNNy}

Recalling that $\mu_{N+2}(B_\be\setminus \overline{B}_\al)=\mu_{2,1}$ by Corollary \ref{cor:multiplicity}, let us prove that
\begin{equation}\label{SW42n}
\mu_{N+2}(\Omega) \leq \mu_{2,1}, 
\end{equation}
assuming that $\Omega$ is symmetric of order $4$. 
Let  $G_2(r)$ be the function defined by Proposition \ref{prop:bound} with $l=2$. Clearly, $G_2(r)$ is at least a $C^1$-function for $r \in (\alpha,+\infty)$.
As an admissible choice of the $(N+1)$-dimensional subspace of $H^1(\Omega)$ for the variational characterization \eqref{eq:mukdef} of $\mu_{N+2}(\Omega)$, we take
$$
X_{N+1} = 
\text{span}
\left\{ 
\frac{G_2(r)}{r}x_1, \dots, \frac{G_2(r)}{r}x_N, w
\right\},
$$
where we define $w$ as an extension to $\Omega$ of a certain $(N+2)$-th eigenfunction of \eqref{eq:N} on $B_\beta \setminus \overline{B}_\alpha$.
Namely, recall that any $(N+2)$-th eigenfunction $\varphi_{N+2}$ of \eqref{eq:N} on $B_\beta \setminus \overline{B}_\alpha$ has the form 
$$
\varphi_{N+2}(x)
=
\frac{v(r)}{r^2}h(x),
$$
where $h\in H_2$ and $v$ is an eigenfunction of the SL problem 
\eqref{eq:ode}, \eqref{eq:bcN} 
associated to $\mu_{2,1}$, see Remark \ref{rem:radial}. 
We use the orthogonal basis $Z_2\cup \widetilde{Z_3}$ of $H_2$ (see Appendix \ref{sec:Orthogonal}), where 
\begin{align*}
    Z_2 
    &=
    \left\{
    x_i x_j:~ 
    i<j \text{ and } i,j = 1,2,\ldots, N \right\},\\  
    \widetilde{Z_3}
    &=
    \left\{
    \frac{1}{\sqrt{i(i+1)}}\left(\sum_{j=1}^ix_j^2 - i x_{i+1}^2\right):~
    i =1,2,\ldots,N-1
    \right\},
\end{align*}
to define the desired function $w$ as follows:
\begin{equation*}\label{eq:ww}
w 
= 
\sqrt{2} \sum_{i=1}^{N-1} \sum_{j=i+1}^N  \frac{G_2(r)}{r^2} x_i x_j
+
\sum_{i=1}^{N-1} \frac{G_2(r)}{\sqrt{i(i+1)}r^2} \left(\sum_{j=1}^i x_j^2 - i x_{i+1}^2\right).
\end{equation*}
Let us remark that in the case $N=2$ the expression for $w$ is reduced to
$$
w = \frac{G_2(r)}{\sqrt{2}r^2}(2 x_1 x_2 + x_1^2-x_2^2).
$$

In view of the symmetry of order $4$, we deduce from Proposition \ref{2Orthogonality00} \ref{2Orthogonality00:1}, \ref{2Orthogonality00:2} (if $N=2$) or
Proposition \ref{Orthogonality00} \ref{Orthogonality00:2} and Proposition \ref{Orthogonality0} \ref{Orthogonality0:3} (if $N \geq 3$) that $\int_\Omega u \,dx = 0$ for any $u \in X_{N+1}$.
Analogously to \eqref{eq:ggg'g'}, we deduce that
\begin{equation}\label{eq:ggg'g'2}
\int_\Omega \left(\frac{G_2(r)}{r}x_i\right)  \left(\frac{G_2(r)}{r}x_j\right) dx=0 
\quad \text{and} \quad
\int_\Omega \nabla \left(\frac{G_2(r)}{r}x_i\right) \nabla \left(\frac{G_2(r)}{r}x_j\right)=0.
\end{equation}
Moreover,  Lemma \ref{lem:z1z2z3-tildeN2} and Remark \ref{rem:lem:z1} (if $N=2$) or Lemma \ref{lem:z1z2z3-tilde} (if $N \geq 3$) also give
\begin{equation}\label{eq:ort-g-w}
\int_\Omega \left(\frac{G_2(r)}{r}x_i\right)  w \,dx=0 
\quad \text{and} \quad
\int_\Omega \nabla \left(\frac{G_2(r)}{r}x_i\right) \nabla w \, dx =0.
\end{equation}
In the same way as in the derivation of \eqref{eq:A}, \eqref{eq:A'}, there exist constants $A_3, A_4>0$ such that for every $i \in \{1,2,\ldots, N\}$ we have
\begin{align}
\int_\Omega \left(\frac{G_2(r)}{r}x_i\right)^2  dx
&=A_3
=
\frac{1}{N}\int_\Omega G_2^2(r) \,dx,\label{eq:A3}\\
\int_\Omega \left|\nabla\left(\frac{G_2(r)}{r} x_i\right)\right|^2dx 
&=A_4
=
\frac{1}{N}\int_{\Omega}
\left((G_2'(r))^2+ \frac{(N-1)G_2^2(r)}{r^2}\right)dx.
\label{eq:A4}
\end{align}

For each $u\in X_{N+1}\setminus\{0\}$ there exist $c_1,c_2, \ldots, c_{N+1} \in \mathbb{R}$, not simultaneously equal to zero, such that
$$
u = c_1 \frac{G_2(r)}{r}x_1 + \dots + c_N \frac{G_2(r)}{r}x_N + c_{N+1} w.
$$
Thus, by the orthogonality \eqref{eq:ggg'g'2}, \eqref{eq:ort-g-w}, and by the expressions \eqref{eq:A3}, \eqref{eq:A4} we obtain
\begin{equation}\label{eq:mu4maxn}
\frac{\int_{\Omega} |\nabla u|^2 \, dx}{\int_{\Omega} u^2 \, dx}
=
\frac{A_4 \displaystyle\sum_{i=1}^{N}c_i^2
	+ c_{N+1}^2 \int_\Omega |\nabla w|^2 \,dx}
{A_3 \displaystyle\sum_{i=1}^{N}c_i^2
	+ c_{N+1}^2 \int_\Omega w^2 \,dx}
\leq
\max\left\{\frac{A_4}{A_3},\frac{\int_\Omega |\nabla w|^2 \,dx}{\int_\Omega w^2 \,dx}\right\},
\end{equation}
and we know from \eqref{eq:A3}, \eqref{eq:A4} that 
\begin{equation*}\label{eq:BA1n}
\frac{A_4}{A_3}
=
\frac{\int_{\Omega}
	\left((G_2'(r))^2+ \frac{(N-1)G_2^2(r)}{r^2}\right)dx}{\int_\Omega G_2(r)^2\,dx }.
\end{equation*}
We claim that 
\begin{equation}\label{eq:claim}
    \frac{\int_\Omega |\nabla w|^2 \,dx}{\int_\Omega w^2 \,dx}=\frac{ \int_{\Omega} \left((G_2'(r))^2 +\frac{2 N G_2^2(r)}{r^2}\right) dx}{\int_\Omega G_2^2(r) \, dx}.
\end{equation}
Suppose we established this claim. 
Then we get from Proposition \ref{prop:bound} with $l=2$ that
\begin{equation}\label{eq:H<Nn+2}
\frac{A_4}{A_3}<\frac{\int_{\Omega} \left((G_2'(r))^2 +\frac{2 N G_2^2(r)}{r^2}\right) dx}{\int_\Omega G_2^2(r) \, dx}
\leq
\mu_{2,1},
\end{equation}
where the second inequality turns to equality if and only if $\Omega$ coincides a.e.\ with $B_\beta \setminus \overline{B}_\al$.
Therefore, by \eqref{eq:mu4maxn} and \eqref{eq:H<Nn+2}, for every $u\in X_{N+1} \setminus \{0\}$ we have
$$
\frac{\int_{\Omega} |\nabla u|^2 \, dx}{\int_{\Omega} u^2 \, dx} \leq \mu_{2,1}.
$$
Now, the Courant-Fischer minimax formula \eqref{eq:mukdef} yields the desired inequality \eqref{SW42n} as follows:
\begin{equation*}
\mu_{N+2}(\Omega) \leq \max_{u \in X_{N+1}\setminus \{0\}} \frac{\int_{\Omega} |\nabla u|^2 \, dx}{\int_{\Omega} u^2 \, dx} \le \mu_{2,1}.
\end{equation*}

Thus, to complete the proof, it remains to establish the claimed equality \eqref{eq:claim}. 
In view of the $L^2(\Omega)$-orthogonality given by 
Lemma \ref{lem:z1z2z3-tildeN2} (if $N=2$, and assuming hereinafter, without loss of generality, that the rotation $T$ is the identity, see Remark \ref{rem:lem:N2}) or Lemma \ref{lem:z1z2z3-tilde} (if $N \geq 3$), we have
\begin{equation*}\label{eq:intw}
\int_\Omega w^2 \,dx 
= \int_\Omega\frac{G_2^2(r)}{r^4}\left[
2 \sum_{i=1}^{N-1} \sum_{j=i+1}^N  x_i^2 x_j^2 \,dx
+
\sum_{i=1}^{N-1} \frac{1}{i(i+1)} \left(\sum_{j=1}^i x_j^2 - i x_{i+1}^2\right)^2 \right] dx,
\end{equation*}
and hence the identity \eqref{eq:alg4} implies that
\begin{align}\label{eq:wnorm}
\int_\Omega w^2 \,dx= \frac{N-1}{N}\int_\Omega G_2^2(r) \,dx.
\end{align}
Next, we calculate $\int_\Omega |\nabla w|^2 \,dx$. 
Again, by the $H^1(\Omega)$-orthogonality given by Lemma \ref{lem:z1z2z3-tildeN2} (if $N=2$) or Lemma \ref{lem:z1z2z3-tilde} (if $N \geq 3$), 
\begin{align*}
\notag
\int_\Omega |\nabla w|^2 \,dx 
&= 
2 \sum_{i=1}^{N-1} \sum_{j=i+1}^N \int_\Omega \left|\nabla \left(\frac{G_2(r)}{r^2} x_i x_j\right)\right|^2 dx
\\
&+
\sum_{i=1}^{N-1}\frac{1}{i(i+1)} \int_\Omega \left|\nabla \left(\frac{G_2(r)}{r^2} \left(\sum_{j=1}^i x_j^2 - i x_{i+1}^2\right)\right)\right|^2 dx.
\end{align*}
Using  the  expressions in  Remark \ref{remark:norm}, we obtain 
\begin{align}
\notag
&\int_\Omega |\nabla w|^2 \,dx 
\\
\notag
&= 
\int_\Omega \left[\frac{(G_2'(r))^2}{r^4} -\frac{4G_2^2(r)}{r^6}\right] 
\left[
2\sum_{i=1}^{N-1}\sum_{j=i+1}^N x_i^2 x_j^2 \,dx
+ \sum_{i=1}^{N-1}\frac{1}{i(i+1)} \left(\sum_{j=1}^i x_j^2 - i x_{i+1}^2\right)^2 \right]dx
\\
\label{eq:intnablaw1}
&+ \int_\Omega \frac{G_2^2(r)}{r^4} \left[2\sum_{i=1}^{N-1}\sum_{j=i+1}^N (x_i^2+ x_j^2)+4\sum_{i=1}^{N-1}\frac{1}{i(i+1)}\left(\sum_{j=1}^i x_j^2 +i x_{i+1}^2\right) \right]dx.
\end{align}
By the identities \eqref{eq:alg2} and \eqref{eq:alg12} we have 
\begin{align}
  2\sum_{i=1}^{N-1}\sum_{j=i+1}^N (x_i^2+ x_j^2)+4\sum_{i=1}^{N-1}\frac{1}{i(i+1)}\left(\sum_{j=1}^i x_j^2 +i x_{i+1}^2\right)
  &=
  2(N-1)r^2+\frac{4
(N-1)}{N}r^2\nonumber\\ &=\frac{2(N-1)(N+2)}{N}r^2.
\label{eq:secondterm}
\end{align} 
Finally, using \eqref{eq:alg4} and \eqref{eq:secondterm}, we conclude from \eqref{eq:intnablaw1} that
\begin{align}
\notag
\int_\Omega |\nabla w|^2 \,dx
&=
\frac{N-1}{N} \int_{\Omega} \left[(G_2'(r))^2 -\frac{4G_2^2(r)}{r^2}\right]dx
+\frac{2(N-1)(N+2)}{N} \int_\Omega \frac{G_2^2(r)}{r^2} \,dx\\
\label{eq:wgradientnorm}
&=
\frac{N-1}{N}\int_{\Omega} \left((G_2'(r))^2 +\frac{2 N G_2^2(r)}{r^2}\right) dx.
\end{align}
Combining now \eqref{eq:wnorm} and \eqref{eq:wgradientnorm}, we get
\begin{equation*}\label{eq:w-rayleigh}
\frac{\int_\Omega |\nabla w|^2 \,dx}{\int_\Omega w^2 \,dx}
=
\frac{\int_{\Omega} \left((G_2'(r))^2 +\frac{2 N G_2^2(r)}{r^2}\right) dx}{\int_\Omega G_2^2(r) \, dx}.
\end{equation*}
This establishes the claimed equality \eqref{eq:claim} and therefore completes the proof of the inequality \eqref{eq:SWNNy}.

\subsection{Proof of \eqref{eq:SWNNz}}

Due to Corollary \ref{cor:multiplicity}, in order to establish \eqref{eq:SWNNz} it is enough to prove that
\begin{equation*}\label{eq:8sym1}
\mu_{5}(\Omega) \leq \mu_{2,1},
\end{equation*}
assuming that $\Omega$ is symmetric of order $8$. 
Let the function $G_2$ be defined by Proposition \ref{prop:bound} with $l=2$. 
As an admissible choice of the $4$-dimensional subspace of $H^1(\Omega)$ for the variational characterization \eqref{eq:mukdef} of $\mu_{5}(\Omega)$, we take
$$
X_4 = 
\text{span}
\left\{ \frac{G_2(r)}{r}x_1, \frac{G_2(r)}{r}x_2, \frac{G_2(r)}{r^2}x_1x_2,\frac{G_2(r)}{r^2}(x_1^2-x_2^2)
\right\}.
$$
Since the symmetry of order $8$ implies the symmetry of order $4$, we get, as in Section \ref{section:SWNNy},
$\int_\Omega u \,dx = 0$ for any $u \in X_{5}$, and equalities \eqref{eq:A3}, \eqref{eq:A4} for $i \in \{1,2\}$ and $N=2$. 
Moreover, Lemma \ref{lem:z1z2z3-tildeN2} (in which we assume, without loss of generality, that the rotation $T$ is the identity, see Remark \ref{rem:lem:N2}) gives the mutual orthogonality of elements of $X_4$ in both $L^2(\Omega)$ and $H^1(\Omega)$.

On the other hand, since $\Omega$ is symmetric of order 8, Proposition \ref{2Orthogonality00} \ref{2Orthogonality00:3} provides the existence of $A_5, A_6>0$ such that
$$
\int_\Omega\frac{G_2^2(r)}{r^4} \left(x_1^2-x_2^2\right)^2 dx
=
4\int_\Omega\frac{G_2^2(r)}{r^4} x_1^2x_2^2\,dx=A_5
$$
and, 
using Remark \ref{remark:norm} with $i=1$, 
\begin{align*}
    \int_\Om \left|\nabla\left( \frac{G_2(r)}{r^2}(x_1^2-x_2^2)\right)\right|^2dx
    &=
    \int_{\Omega} \left[\left(\frac{(G_2'(r))^2}{r^4}-\frac{4G_2^2(r)}{r^6}\right)(x_1^2-x_2^2)^2+\frac{4G_2^2(r)}{r^2}\right]dx\\
    &=
    4\int_{\Omega} \left[\left(\frac{(G_2'(r))^2}{r^4}-\frac{4G_2^2(r)}{r^6}\right)x_1^2x_2^2+\frac{G_2^2(r)}{r^2}\right]dx\\
    &=
    4\int_\Om \left|\nabla\left( \frac{G_2(r)}{r^2}x_1x_2\right)\right|^2dx=A_6.
\end{align*}
This yields
\begin{align}
    2A_5&=\int_\Omega\frac{G_2^2(r)}{r^4} \left(x_1^2-x_2^2\right)^2 dx+4\int_\Omega\frac{G_2^2(r)}{r^4} x_1^2x_2^2\,dx
    =
    \int_\Om G_2^2(r) \,dx,\label{eq:8sym3}\\
    2A_6&=\int_{\Omega} \left[\left(\frac{(G_2'(r))^2}{r^4}-\frac{4G_2^2(r)}{r^6}\right)r^4+\frac{8G_2^2(r)}{r^2}\right]dx
    =
    \int_{\Omega} \left((G_2'(r))^2+\frac{4G_2^2(r)}{r^2}\right)dx.\label{eq:8sym4}
\end{align}
Thus, for any $u \in X_4 \setminus \{0\}$, by the orthogonality given by Lemma \ref{lem:z1z2z3-tildeN2}, and by the expressions \eqref{eq:A3}, \eqref{eq:A4}, and \eqref{eq:8sym3}, \eqref{eq:8sym4}, we get 
$$
\frac{\int_\Omega|\nabla u|^2 \,dx }{\int_\Omega u^2 \,dx}
\le 
\max\left\{\frac{A_4}{A_3},\frac{A_6}{A_5}\right\}
=
\frac{\int_{\Omega} \left((G_2'(r))^2 +\frac{4 G_2^2(r)}{r^2}\right)dx}{\int_{\Omega}G_2^2(r)\,dx}.
$$
Therefore, by the Courant-Fischer minimax formula \eqref{eq:mukdef} and Proposition \ref{prop:bound} with $l=2$, we conclude that
\begin{align*}
\mu_5(\Omega_{\text{out}} \setminus \overline{\Omega}_{\text{in}}) 
\leq 
\max_{u \in X_4\setminus \{0\}} \frac{\int_{\Omega} |\nabla u|^2 \, dx}{\int_{\Omega} u^2 \, dx}
\leq
\mu_{2,1},
\end{align*}
where the second inequality turns to equality if and only if $\Omega$ coincides a.e.\ with $B_\beta \setminus \overline{B}_\al$.
This completes the proof of \eqref{eq:SWNNz}.

\section{Counterexamples}\label{sec:counterexample}

In this section, we show that the
inequalities \eqref{eq:SWNN},  \eqref{eq:SWNNx},  \eqref{eq:SWNNy}, and \eqref{eq:SWNNz} stated in Theorem \ref{thm1} might fail for domains which do not satisfy the corresponding symmetry requirements. 
For simplicity, all the examples will be given in the planar case $N=2$.

\subsection{Counterexample to \eqref{eq:SWNN}} We 
consider the class of eccentric annuli $B_\beta \setminus \overline{B_\alpha(s)}$ with $0<\alpha<\be$ and $s \in (0,\beta-\alpha)$, where $B_\alpha(s)$ is the open disk of radius $\alpha > 0$ centred at the point $(s,0)$. 
Clearly, for $s>0$, $B_\beta \setminus \overline{B_\alpha(s)}$ is neither symmetric of order $2$ nor centrally symmetric.
The fact that  $\mu_2(B_\beta \setminus \overline{B_\alpha(s)})$ can be greater than $\mu_2(B_\beta \setminus \overline{B_\alpha(0)})$ for certain values of $\al,\be, s$  is observed numerically in \cite{yee}, see \cite[Figures 7, 8]{yee}.
More rigorously, taking, for instance, $\alpha=s=0.25$ and $\beta=1$, we have $\mu_2(B_\beta \setminus \overline{B_\alpha(0.25)}) \geq (1.6446...)^2$, while $\mu_2(B_\beta \setminus \overline{B_\alpha(0)}) = (1.6445...)^2$, see \cite[Table IX]{kuttler}. This establishes a counterexample to the inequality \eqref{eq:SWNN}.

\subsection{Counterexamples to \eqref{eq:SWNNx} and \eqref{eq:SWNNy}}\label{counter:3}
 Consider the rectangle $\Omega_{\text{out}} = \left(-\frac{a}{2},\frac{a}{2}\right)\times \left(-\frac{1}{2a},\frac{1}{2a}\right)$ of unite measure.
 Clearly, $\Omega_{\text{out}}$ is not symmetric of order $4$ provided $a \neq 1$.
 It is well-known that all eigenvalues of the problem \eqref{eq:N} on $\Omega_{\text{out}}$ are given by $\frac{\pi^2 k^2}{a^2}+\pi^2 m^2 a^2$, $k,m \in \mathbb{N}_0$.
 Taking $a=\sqrt{3}$, it is not hard to deduce that 
 \begin{align}
 \label{eq:mu3}
 \mu_3(\Omega_{\text{out}}) = \frac{4\pi^2}{3} \approx 13.1594 > \mu_2(B) = \pi \left(p_{1,1}^{(1)}\right)^2 \approx 10.6499,\\
 \label{eq:mu4}
 \mu_4(\Omega_{\text{out}}) = 3\pi^2 \approx 29.6088 > \mu_4(B) = \pi \left(p_{1,1}^{(2)}\right)^2 \approx 29.3059,
 \end{align}
 where $B$ is a disk of unit measure. These inequalities show that \eqref{eq:SWsym} and \eqref{eq:SWsym4} are not satisfied for $\Omega_{\text{out}}$.	
 On the other hand, it is known that
 \begin{equation}\label{eq:counter2}
 \mu_\kappa(\Omega \setminus \overline{B}_{\alpha}) \to \mu_\kappa(\Omega)
 \quad \text{as}~ \alpha \to 0
 \end{equation}
 for any $\kappa \in \mathbb{N}_0$, see, e.g., \cite[Theorem 3.5 and Corollary 3.6]{daners}. 
 Therefore, 
 combining \eqref{eq:mu3} (reps.~\eqref{eq:mu4}) and \eqref{eq:counter2} (with $\Omega=\Omega_{\text{out}}$ and $\Omega=B_\beta$), we provide a counterexample to \eqref{eq:SWNNx} (resp.~\eqref{eq:SWNNy}) for all sufficiently small $\alpha \geq 0$.

\subsection{Counterexample to \eqref{eq:SWNNz}}
It was observed by \textsc{Hersch} in \cite[Section 5.4]{hersh2} that the inequality \eqref{eq:SWsym4} cannot be extended to the inequality 
$$
\mu_5(\Omega) \leq \mu_4(B)
$$
for domains symmetric of order $4$, 
since this inequality is reversed when $\Omega$ is a square.
Clearly, using the same convergence argument as in the counterexample above, we deduce that the inequality \eqref{eq:SWNNz} also does not hold, in general, if $\Omega_{\textnormal{out}} \setminus \overline{B}_\al$ has only the symmetry of order $4$.

	\section{On the existence and properties of symmetric domains}\label{sec:existence}

In this section, we discuss the existence and certain properties of domains with symmetries as required in Theorem \ref{thm1} and Remark \ref{rem:nonexistence-symmetry}. 
Since radially symmetric domains trivially satisfy all such symmetries, in what follows we will be interested only in nonradial domains. 
It is evident that any domain in any dimension is symmetric of order $1$. 
In the planar case $N=2$, 
nonradial domains with symmetry of order $q$ exist for any $q \geq 2$. 
For example, a regular $q$-sided polygon is symmetric of order $q \geq 3$, and a square is symmetric of order $2$.
Thus, throughout this section, we will be interested mainly in the case $N \geq 3$ and $q \geq 2$. 

Recall that for each $q\in \N$ we denote by $R^{2\pi/q}_{i,j}$ the $2\pi/q$-rotation (in the anti-clockwise direction with respect to the origin) in the coordinate plane $(x_i,x_j)$ with $i<j$, see Definition \ref{def}.
In particular, we will be interested in the cases $q=2,4,8$:
    \begin{align}
    \label{eq:R2}
    &R^{2\pi/2}_{i,j}(x_1,\ldots, x_i,\ldots, x_j,\ldots, x_N)
    =
    (x_1,\ldots, -x_i, \ldots, -x_j,\ldots, x_N),\\
    \label{eq:R4}
    &R^{2\pi/4}_{i,j}(x_1,\ldots, x_i,\ldots, x_j,\ldots, x_N)
    =
    (x_1,\ldots, -x_j, \ldots, x_i,\ldots, x_N),\\
    \label{eq:R8}
    &R^{2\pi/8}_{i,j}(x_1,\ldots, x_i,\ldots, x_j,\ldots, x_N)
    =
    \left(x_1,\ldots, \frac{1}{\sqrt{2}}(x_i-x_j), \ldots, \frac{1}{\sqrt{2}}(x_i+x_j),\ldots, x_N\right).
    \end{align}
The antipodal map will be denoted by $A,$ i.e., $A(x)=-x$.

For $x \in \mathbb{R}^N$, we denote by $\|x\|_p$ the standard $L^p$-norm of $x$, that is, $\|x\|_p := (|x_1|^p + \dots + |x_N|^p)^{1/p}$.

\subsection{Symmetry of order $q$}\label{subsec:q}
Let $\Om$ be a domain which has the symmetry of order $q \in \mathbb{N}$. 
If $q \neq 1,2,4$, then we claim that $\Om$ must be invariant under the entire group $SO(N)$, and hence $\Omega$ is radially symmetric. 
Here, $SO(N)$ stands for the special orthogonal group consisted of all $N \times N$ orthogonal matrices of determinant $1$, which correspond to rotations in $\mathbb{R}^N$.
We denote by $\|T\|$ the operator norm of $T \in SO(N)$ induced by the Euclidean norm $\|\cdot\|_2$. 
The group $SO(N)$ is a real compact Lie group. 

Let us fix some additional notation.
\begin{itemize}
     \item For a subgroup $H$ of $SO(N)$ and for $1\le i<j<k\le N,$ $$
     H_{i,j,k}
     :=
     \left\{
     T \in H:~ 
     T(x)=x ~\text{for every}~ x\in \text{span}\{e_i,e_j,e_k\}^{\perp}\right\}.
     $$
    \item $G_q \subset SO(N)$ is the group generated by the set of rotations $\left\{R_{i,j}^{2\pi/q}: 1\le i<j\le N\right\}$.
\end{itemize}

\begin{proposition}
  Let $\Om$ be a domain in $\mathbb{R}^N$ with $N\ge 3$.
  Let 
  $H:=\{T \in SO(N),$ $T(\Omega)= \Omega\}$. 
  If a natural number $q \ne 1, 2,4$ and $G_q \subset H$, then $H=SO(N)$. 
\end{proposition}
\begin{proof}
We divide the proof into three steps.

\noindent {\bf Step 1.} 
We claim that $H$ is a closed subgroup of $SO(N)$.
First, it is evident that $H$ is a subgroup of $SO(N)$.
Let now $\{T_n\}_{n \in \mathbb{N}}\subset H$ and $T\in SO(N)$ be such that $T_n \to T$, where the convergence is understood in  the operator norm. 
Since $\Omega$ is open, for any $x \in \Omega \setminus\{0\}$ we can choose $\varepsilon>0$ such that $B_\varepsilon(x)\subset \Omega$.
Let $n$ be sufficiently large so that $\|T_n-T\|<\frac{\varepsilon}{\|x\|_2}$.
Thus,
$\|T_n(x)-T(x)\|_2 \le \|T_n-T\|\|x\|_2<\varepsilon$. Moreover, since $T_n$ is an isometry, we have $B_\varepsilon(T_n(x))=T_n(B_\varepsilon(x))$.
Therefore, $T(x)\in B_\varepsilon(T_n(x))=T_n(B_\varepsilon(x))\subset T_n(\Om)=\Om$, which yields $T(\Om)\subset \Om$. 
By the same argument, we also obtain $T^{-1}(\Om)\subset \Om$, and hence $\Om\subset T(\Om)$. 
Consequently, $T(\Omega)=\Omega$, that is, $T\in H$, which implies that $H$ is closed. 

\noindent {\bf Step 2:} 
We claim that the  proposition is true in the case $N=3$.
Notice that the closed-subgroup theorem asserts that $\overline{G_q}$, the closure  of $G_q$, is a Lie subgroup of $SO(3)$.
Since $q\ne 1,2,4$, \cite[Theorem 1, Corollary 2, and Remark on pp.~613-614]{RS} imply that  $G_q$ contains an infinite subgroup which is a free product or a nontrivial amalgamated free product of finite groups, and hence $G_q$ must be also an infinite group.
Therefore, according to the classification of Lie subgroups of $SO(3)$ (see, e.g., \cite[Example 2.4]{AFG}), $\overline{G_q}$ is either $SO(2)$ or $O(2)$ or $SO(3)$.
Since $G_q$ contains two rotations by angles other than $\pi$ degrees about two different axes, we conclude that $\overline{G_q}$ is neither $SO(2)$ nor $O(2)$, and hence $\overline{G_q}=SO(3)$. 
Since $G_q \subset H$, and $H$ is a closed subgroup of $SO(N)$ by Step 1, we deduce that $H=SO(3)$.

\noindent \noindent {\bf Step 3:} 
Finally, we claim that $H=SO(N)$ for any $N \ge 4$.
By the same set of arguments as in Step 2, for each $1\le i<j<k\le N$, we obtain $H_{i,j,k}=SO(N)_{i,j,k}$.
Notice that the Lie algebra of  $SO(N)$ is the set all $N\times N$ skew-symmetric matrices.
Now take skew-symmetric matrices $E_{i,j}$ defined as  $E_{i,j}(e_l)=\delta_{i,l}e_j-\delta_{j,l}e_i$. 
Clearly, $E_{i,j}$ (as well as $E_{i,k}$ and $E_{j,k}$)
lies in the Lie algebra of $H_{i,j,k}$ 
and hence in the Lie algebra of $H$.
Thus, the Lie algebra of $H$ contains the set  $\{E_{i,j}: 1\le i<j<k\le N\}$. 
Since the set $\{E_{i,j}: 1\le i<j<k\le N\}$ spans the set of all $N \times N$ skew-symmetric matrices, we conclude that $H=SO(N)$.
\end{proof}

\subsection{Symmetry of order $2$ and central symmetry}\label{subsec:2}

Observe that in the planar case $N=2$ the symmetry of order $2$ is equivalent to the central symmetry. 
Let us discuss the relation between these two notions for $N\ge 3$. 
For this, taking $p\in [1,\infty]$, we consider the following four sets
\begin{equation}\label{eq:Q}
Q_p
:=
\{x \in \mathbb{R}^N:~ \|x\|_{p} < 1\}, \quad 
Q_p^+
:=
\{x \in Q_p:~ x_1,\dots,x_N > 0\},
\quad
Q_p^-:=-Q_p^+,
\end{equation}
\begin{equation*}
\widetilde{Q}_p^+ 
:= \left\{
x \in Q_p:~ x_1 \cdot \ldots \cdot x_N > 0
\right\},
\end{equation*}
see Figures \ref{fig:1} and \ref{fig:2}.
Notice that the $L^p$-cube $Q_p$ has both the symmetry of order $2$ and the central symmetry, see \eqref{eq:R2} and the definition of the antipodal map $A$.
Moreover, if $p \neq 2$, then $Q_p$ is nonradial.

Assume first that $N \geq 4$ is \textit{even}.
In this case, the domains that are symmetric of order $2$ must be centrally symmetric, as it  follows from the fact that the antipodal map $A$ can be expressed as $A=\prod_{i=1}^{N/2}R^{2\pi/2}_{2i-1,2i}$.

Assume now that $N \geq 3$ is \textit{odd}.
In this case, the symmetry of order $2$ might not imply the central symmetry. 
As an example, 
 for  $\alpha \in (0,1)$, the set $\widetilde{Q}_p^+ \cup
B_\alpha$ is a bounded domain that has the symmetry of order $2$, see \eqref{eq:R2}.
However, $A(\widetilde{Q}_p^+ \cup
B_\alpha) \neq \widetilde{Q}_p^+ \cup
B_\alpha$ since $N$ is odd. 
See Figure \ref{fig:2}.

\begin{figure}[htb]
    \begin{minipage}[t]{.45\textwidth}
        \centering
        \includegraphics[width=0.7\textwidth]{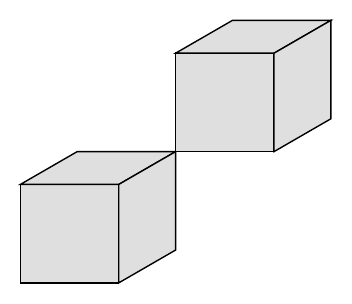}
        \caption{$Q_\infty^+ \cup Q_\infty^-$ for $N=3$.}\label{fig:1}
    \end{minipage}
    \hfill
    \begin{minipage}[t]{.45\textwidth}
        \centering
        \includegraphics[width=0.7\textwidth]{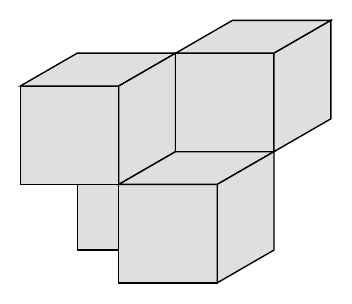}
        \caption{$\widetilde{Q}_\infty^+$ for $N=3$.}\label{fig:2}
    \end{minipage} 
\end{figure}

\medskip
In the following lemma, we demonstrate that, in general, the central symmetry does not imply the symmetry of order $2$ in any dimension $N \geq 3$.
Here, we will denote by $B_\alpha(x)$ an open $N$-ball of radius $\alpha$ centred at $x \in \mathbb{R}^N$.

\begin{lemma}\label{lem:sym2}
	Let $N \geq 3$ and $p=\infty$.
	Let $\xi=(1,0,\dots,0)$.
	Then, for any sufficiently small $\alpha \in (0,1)$, $Q_p^+\cup Q_p^-\cup B_\al(0) \cup B_\al(\xi) \cup B_\al(-\xi)$ is a domain which has the central symmetry, but does not have the symmetry of order 2.
\end{lemma}
\begin{proof}
    Let $\Omega:=Q_p^+\cup Q_p^-\cup B_\al(0) \cup B_\al(\xi) \cup B_\al(-\xi)$.
	It is easy to see that $\Omega$ is a centrally symmetric domain.
	Let us prove that $\Omega$ is not symmetric of order $2$. 
	Suppose, by contradiction, that  there exists an  isometry $T$ such that $R^{2\pi/2}_{i,j} T(\Omega) = T(\Omega)$ for every $1\le i<j \leq N$. 
	Since $\Omega$ is centrally symmetric and $T(\Omega)$ is symmetric of order $2$, we deduce from Proposition \ref{Orthogonality00} that their centroids (centres of mass) coincide with the origin. 
	Since the isometry $T$ maps the centroid of $\Omega$ to the centroid of $T(\Omega)$, we deduce that $T$ must be  an orthogonal transformation. 
	
	Evidently, $\Om$ is  symmetric of order $2$ if and only if $\overline{\Omega}$ is symmetric of order $2$. 
 	Since $p=\infty$ and $\alpha \in (0,1)$ is small enough, we have $\|x\|_2 \leq \sqrt{N}$ for any $x \in \overline{\Omega}$,  and equality holds if and only if  $x=\pm u$, where $u=(1,\dots,1)$.
 	Since there are only two such points and $T$ is an orthogonal transformation, for each $1\le i<j \leq N$ the transformation  $T^{-1} R^{2\pi/2}_{i,j} T$ either fixes $u$, or maps $u$ to $-u$.
	Let us denote $v=T(u)$ and show that either $v=\sqrt{N}e_i$ or $v=-\sqrt{N}e_i$ for some $i \in \{1,\dots,N\}$. 
 	In other words, the axis $e_i$ is mapped by $T^{-1}$ to the  diagonal $\{t u: t \in \mathbb{R}\}$
	of the cube $Q_p$.
	Denote 
	\begin{align*}
		S_+ &=
		\left\{
		R^{2\pi/2}_{i,j}:~ 
		1 \leq i < j \leq N,~ 
		R^{2\pi/2}_{i,j} v = v
		\right\},\\
		S_- &=
		\left\{
		R^{2\pi/2}_{i,j}:~ 
		1 \leq i < j \leq N,~ 
		R^{2\pi/2}_{i,j} v = -v
		\right\}.
	\end{align*}
	Notice that, if $R^{2\pi/2}_{i,j} \in S_+$, then $v_i=0$ and $v_j=0$, while if  $R^{2\pi/2}_{i,j} \in S_-$, then $v_k=0$ for any $k \neq i,j$, see \eqref{eq:R2}. 
	Since $v\ne 0$, we have  $S_+,S_- \neq \emptyset$. 
	Assume, without loss of generality, that $R^{2\pi/2}_{1,2}\in S_{-}$. 
	Now, if $R^{2\pi/2}_{1,3}\in S_{-}$, then we get
	either $v=\sqrt{N}e_1$ or $v=-\sqrt{N}e_1$, while if $R^{2\pi/2}_{1,3}\in S_{+}$, then we get either  $v=\sqrt{N}e_2$ or $v=-\sqrt{N}e_2$. 
	For convenience, assume that $v=\sqrt{N}e_1$.
	In particular, we have
	$R_{1,j}^{2\pi/2}v=-v$ for any $1<j \leq N$, and $R_{i,j}^{2\pi/2}v=v$ for any $1< i < j \leq N$.

	Consider now the point $\omega = (1+\alpha,0,\dots,0)$. It is easy to see that $\omega$ and $-\omega$ are the  only  points on $\partial \Omega$ 
	with the properties that their Euclidean norm is $1+ \alpha$ and the mean curvature of $\partial \Omega$ at these points is a nonzero constant. Consequently, we must have $T^{-1}R_{i,j}^{2\pi/2}T(\omega)=\pm\omega$ for any $1 \leq i < j\le N$.
 	Recalling now that $R_{i,j}^{2\pi/2}T(u)=T(u)$ for any $1<i<j\le N$, we obtain
 	 $$
 	 0<\omega\cdot u= R_{i,j}^{2\pi/2}T(\omega)\cdot R_{i,j}^{2\pi/2}T(u)=R_{i,j}^{2\pi/2}T(\omega)\cdot T(u)=T^{-1}R_{i,j}^{2\pi/2}T(\omega)\cdot u.
 	 $$
	Therefore, $T^{-1}R_{i,j}^{2\pi/2}T(\omega)=\omega$ for any $1<i<j\le N$, which yields $T(\omega)_i = 0$ for  $i \in \{2,\dots,N\}$. 
	Thus, we have $T(\omega)=(1+\al)e_1$, and hence
	$$
	w=(1+\al)T^{-1}(e_1)=(1+\al)\frac{1}{\sqrt{N}}u,
	$$
	 which is impossible.
\end{proof}

In conclusion, when $N \geq 4$ is even, the symmetry of order $2$ implies the central symmetry, but not vice versa, i.e., the central symmetry is a weaker notion.
When $N \geq 3$ is odd, the symmetry of order $2$ and the central symmetry are independent notions. 

\begin{figure}[h!]
	\center{\includegraphics[width=0.35\linewidth]{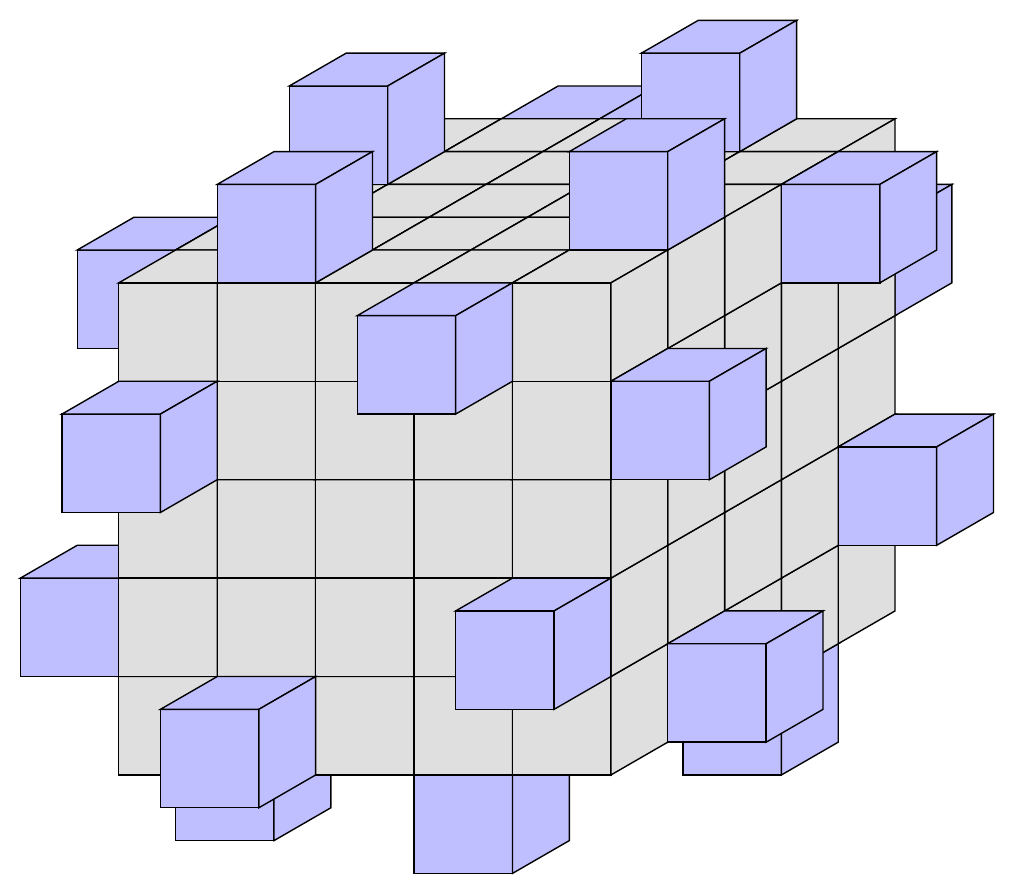}}
	\caption{A set in $N=3$ which is symmetric of order $4$ but not centrally symmetric.}
	\label{fig:3}
\end{figure}

\subsection{Symmetry of order $4$}
For $p\in[1,\infty]$ with $p\ne 2$,  $Q_p$ defined in \eqref{eq:Q} is a nonradial domain having the symmetry of order $4$. 
Indeed, applying any rotation $R_{i,j}^{2\pi/4}$ (see \eqref{eq:R4}), we deduce that $\|R_{i,j}^{2\pi/4} (x)\|_{p} = \|x\|_{p}$, and hence $R_{i,j}^{2\pi/4} (Q_p) \subset Q_p$. 
This implies that $R_{i,j}^{2\pi/4} (Q_p) = Q_p$ since $R_{i,j}^{2\pi/4}$ is an isometry.

It is interesting to mention that the symmetry of order $4$ might not imply the central symmetry if $N$ is odd, as indicated by an example for $N=3$ depicted in Figure \ref{fig:3}.

\subsection{Symmetry of order $8$}
Consider the class of domains which are symmetric of order $8$ with respect to \textit{some} coordinate plane $(x_k,x_l)$ and symmetric of order $4$ with respect to all other coordinate planes, as was discussed in Remark \ref{rem:nonexistence-symmetry}. 
Let us denote this class as $C_{8,4}$.
We show that $C_{8,4}$ coincides with the class of domains symmetric of order $8$ (with respect to all coordinate planes), and hence any $\Omega \in C_{8,4}$ is radially symmetric, see Section \ref{subsec:q}.
Indeed, let us take any $\Omega \in C_{8,4}$, fix any $m \neq k, l$, and show that $\Omega$ is symmetric of order $8$ with respect to $(x_l,x_m)$.
Assuming, for the sake of clarity, that $1<k<l<m<N$, we deduce that for any $x \in \Omega$, 
\begin{align*}
&R_{k,m}^{2\pi/4}(x_1,\dots,x_k,\dots,x_l,\dots,x_m,\dots,x_N) 
= 
(x_1,\dots,-x_m,\dots,x_l,\dots,x_k,\dots,x_N),
\\
&R_{k,l}^{2\pi/8}(x_1,\dots,-x_m,\dots,x_l,\dots,x_k,\dots,x_N) 
\\
&= 
\left(x_1,\dots, -\frac{1}{\sqrt{2}}(x_l+x_m), \dots, \frac{1}{\sqrt{2}}(x_l-x_m),\dots,x_k,\dots, x_N\right),
\\
&\left(R_{k,m}^{2\pi/4}\right)^{-1}\left(x_1,\dots, -\frac{1}{\sqrt{2}}(x_l+x_m), \dots, \frac{1}{\sqrt{2}}(x_l-x_m),\dots,x_k,\dots, x_N\right)
\\
&=
\left(x_1,\dots, x_k, \dots, \frac{1}{\sqrt{2}}(x_l-x_m),\dots,\frac{1}{\sqrt{2}}(x_l+x_m),\dots, x_N\right)
\\
&=R_{l,m}^{2\pi/8}(x_1,\dots,x_k,\dots,x_l,\dots,x_m,\dots,x_N).
\end{align*}
That is, $\left(R_{k,m}^{2\pi/4}\right)^{-1}R_{k,l}^{2\pi/8}R_{k,m}^{2\pi/4}(x) = R_{l,m}^{2\pi/8}(x)$, and hence $R_{l,m}^{2\pi/8}(x) \in \Omega$, which establishes the claim.

\section{Final comments and open problems}\label{sec:final}

Let us list some of the natural questions and remarks related to the discussion of this paper. 

\begin{enumerate}
	\item 
	In the planar case $N = 2$, is it possible to obtain similar inequalities as in Theorem \ref{thm1} for eigenvalues with higher indices ($\geq 6$) by imposing higher symmetry assumptions on the domain?
	Notice that in our proof we  used the constant extensions of \textit{nonradial} eigenfunctions of \eqref{eq:N} on $B_\beta \setminus \overline{B}_\al$.
	However, there are also  \textit{radial} eigenfunctions, which might disturb the ``proper'' index counting.
	For instance, if $N=2$ and $\alpha \geq 0$ is sufficiently small, then already $\mu_{6}(B_\beta \setminus \overline{B}_\alpha)$ might have a one-dimensional eigenspace of radial eigenfunctions.
	\item In the higher-dimensional case $N \geq 3$, is it possible to establish results similar to those of Theorem \ref{thm1} for domains with other symmetries  (e.g., symmetry groups of platonic solids except of the symmetry of order $4$) or for domains of the form $\Om_1 \times \Om_2$, where  $\Om_1$ and $\Om_2$ are domains symmetric of order $q_1$ and $q_2$, respectively?
    In the case of surfaces of platonic solids in $\mathbb{R}^3$, we refer the reader to \cite{EP3,hersch3} for estimates on $\mu_2$ and further discussion. (Notice that the only platonic solid with unknown effective lower bound on $\mu_2$ is the dodecahedron.)
	\item It is natural to wonder which domains maximize higher Neumann eigenvalues (with index $\geq 4$) if no symmetry assumptions are imposed. 
	In this regard, we refer the reader to \cite{Abele,AO1} for numerical results on the maximization of higher Neumann eigenvalues with respect to domains of equal volume.
	\item The counterexamples to \eqref{eq:SWNNx}, \eqref{eq:SWNNy}, \eqref{eq:SWNNz} obtained in Section \ref{sec:counterexample} show \textit{certain} optimality of the obtained results in the planar case. However, the actual optimality would follow from the construction of counterexamples to these inequalities in the classes of domains whose order of symmetry is reduced by one, i.e., for domains symmetric of order $3$ in the case of \eqref{eq:SWNNx}, \eqref{eq:SWNNy}, and of order $7$ in the case of \eqref{eq:SWNNz}. 
	(Recall that \eqref{eq:SWsym} \textit{is} satisfied for planar domains symmetric of order $3$ under the additional simply-connectedness assumption.)
	\item Recall that the assumption that $B_\alpha$ must be contained in $\Omega_{\text{in}}$ appears only in the proof of Proposition \ref{prop:bound}. 
	Can this assumption be relaxed? 
	In particular, one could wonder whether $B_\alpha \subset \Omega_\text{in}$ can be replaced by the equality of measures $|B_{\alpha}|= |\Omega_\text{in}|$. 
	Numerical experiments with domains depicted on Figure \ref{fig:contr} indicate that if $\Omega_\text{in}$ is not connected (and does not contain zero), then the assumption $|B_{\alpha}|= |\Omega_\text{in}|$
	might not lead to the inequality \eqref{eq:SWNN}. 
\begin{figure}[h!]
	\center{\includegraphics[width=0.45\linewidth]{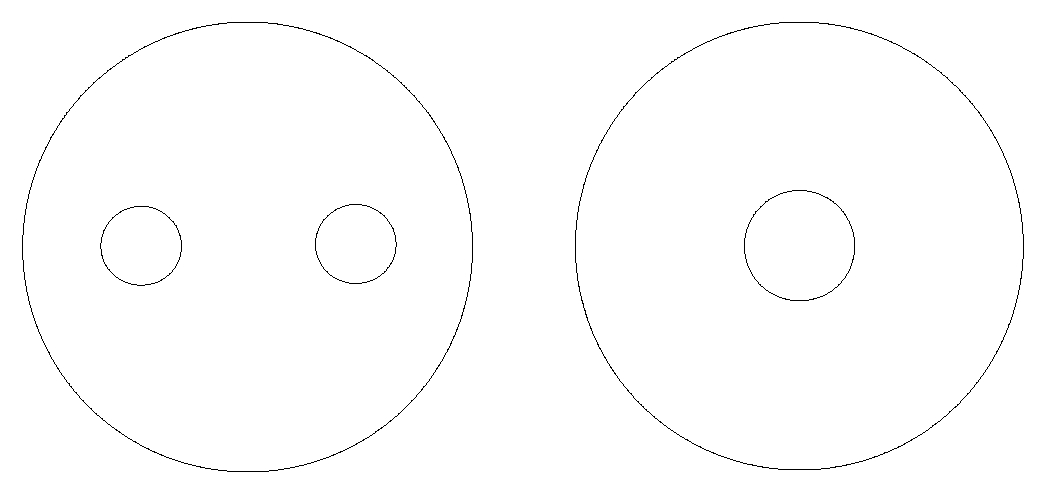}}
	\caption{On the left: a domain $\Omega$ with three circular boundaries. The outer disk is of radius $1$, the inner disks are of radius $\sqrt{2}/8$, the distance between centres of the inner disks is $3/5+\sqrt{2}/4$, $\Omega$ is symmetric of order $2$, and $\mu_2(\Omega) \approx 2.74$.
	On the right: $B_\beta \setminus \overline{B}_\alpha$ with $\beta=1$, $\alpha=1/4$, and $\mu_2(B_\beta \setminus \overline{B}_\alpha) \approx 2.70$.}
	\label{fig:contr}
\end{figure}	
	
	\item 
	If $\Omega$ has an appropriate smoothness (say, Lipschitz), then  equality holds in \eqref{eq:SWNN}, \eqref{eq:SWNNx}, \eqref{eq:SWNNy}, \eqref{eq:SWNNz} if and only if $\Omega = B_\beta \setminus \overline{B}_\al$.
	The necessary part follows since
	$|\Omega| = |B_\beta \setminus 
\overline{B}_\alpha|$, and
	$\Omega \subset B_\beta \setminus 
\overline{B}_\alpha$ in case of equality, see the end of the proof of Proposition \ref{prop:bound}.
\end{enumerate}

\appendix
\section{}\label{sec:appendix2}

In this section, we prove Lemmas \ref{lem:n+2} and \ref{lem:v}  stated in Section \ref{sec:spectrum}.
Throughout the section, we always assume $0 \leq \alpha < \beta$.
First, we prove Lemma \ref{lem:v}.
\begin{proof}[Proof of Lemma \ref{lem:v}]
	We start by showing that $v'$ does not vanish on $(\al, \be).$
	For convenience, let us rewrite \eqref{eq:ode} as
	\begin{equation}\label{eq:ode2}
	v''(r)+
	\frac{N-1}{r} v'(r)
	=
	\left(\frac{l(l+N-2)}{r^2}-\mu_{l,1}\right)v(r), 
	\quad  r\in (\al,\be).
	\end{equation}
	Suppose, by contradiction, that $v'(\gamma)=0$ for some $\ga\in (\al,\be)$, and hence
	\begin{equation*}
	v''(\ga)
	=
	\left(\frac{l(l+N-2)}{\ga^2}-\mu_{l,1}\right)v(\ga).
	\end{equation*}
	Assume first that $v''(\ga)\ge 0$.
	Since $v$ is positive, we have $\frac{l(l+N-2)}{\ga^2}-\mu_{l,1}\ge 0$, and hence, multiplying \eqref{eq:ode2} by $r^{N-1}$, we obtain
	$$
	(r^{N-1}v'(r))'=\left(\frac{l(l+N-2)}{r^2}-\mu_{l,1}\right)v(r)r^{N-1}>0, \quad  r\in (\al,\ga).
	$$ 
	That is, $r^{N-1}v'(r)$ is strictly increasing on $(\al,\ga)$, which is impossible since $v'(\al)=	0$. 
	
	Assume now that $v''(\ga)< 0$.
	In this case, we get $\frac{l(l+N-2)}{\ga^2}-\mu_{l,1}< 0.$ 
	Thus, we see from \eqref{eq:ode2}, as above, that
	$$
	(r^{N-1}v'(r))'=\left(\frac{l(l+N-2)}{r^2}-\mu_{l,1}\right)v(r)r^{N-1}<0, \quad  r\in [\ga,\be).
	$$ 
	Consequently, $r^{N-1}v'(r)$ is strictly decreasing on $[\ga,\be)$, which gives a contradiction since $v'(\be)=0$.
	
	Let us show now that $v'>0$ on $(\alpha,\beta)$.
	Since $v'$ does not vanish in $(\al,\be)$, we have either $v'(r)>0$ or $v'(r)<0$ for every $r\in(\al,\beta).$
	Suppose, by contradiction, that $v'(r)<0$  for every $r\in(\al,\beta)$. 
	In particular, this yields  $v(\al)>0$ and $v''(\al)\le 0$, since $v$ is positive and $v'(\al)=0$. 
	Therefore, from \eqref{eq:ode2} we obtain
	\begin{equation*}
	0 \geq 
	v''(\al)
	=
	v''(\al)+ \frac{N-1}{\al}v'(\al) = \left(\frac{l(l+N-2)}{\al^2}-\mu_{l,1}\right)v(\al),
	\end{equation*}
	and hence $\frac{l(l+N-2)}{\al^2}-\mu_{1,1} \le 0.$ 
	Again from \eqref{eq:ode2} we get
	\begin{equation*}
	(r^{N-1}v'(r))'
	=
	\left(\frac{l(l+N-2)}{r^2}-\mu_{l,1}\right)v(r)r^{N-1}
	<
	\left(\frac{l(l+N-2)}{\al^2}-\mu_{l,1}\right)v(r)r^{N-1} \le 0
	\end{equation*}
	for any $r \in (\alpha,\beta)$.
	Thus, $r^{N-1}v'(r)$ is strictly decreasing on $(\al,\be)$, which is impossible since  $v'(\be)=0$.
	Hence, we proved that $v'(r)>0$ for every $r\in(\al,\beta)$.
	
	Let us now prove the inequality \eqref{eq:long}.
	Multiplying \eqref{eq:ode2} by $v$, we get
	$$
	v''(r)v(r) = - \frac{N-1}{r} v'(r)v(r) + \left(\frac{l(l+N-2)}{r^2} - \mu_{l,1}\right) v^2(r),
	\quad  r \in (\alpha,\beta).
	$$
	Recalling that $v$ is positive, increasing, and $v'(\beta)=0$, we see that $v''(\beta) \leq 0$, which yields
	$$
	v''(\beta)v(\beta) =  \left(\frac{l(l+N-2)}{\beta^2} - \mu_{l,1}\right) v^2(\beta) \leq 0.
	$$
	If $\frac{l(l+N-2)}{\beta^2} - \mu_{l,1} = 0$, then $\frac{l(l+N-2)}{r^2} - \mu_{l,1} > 0$ for all $r \in (\alpha,\beta)$, and hence  \eqref{eq:long} follows trivially.
	On the other hand, if $\frac{l(l+N-2)}{\beta^2} - \mu_{l,1} < 0$, then 
	$$
	0 \geq \left(\frac{l(l+N-2)}{r^2} - \mu_{l,1}\right) v^2(r) 
	> 
	\left(\frac{l(l+N-2)}{\beta^2} - \mu_{l,1}\right) v^2(\beta)
	$$
	for all those $r \in (\alpha,\beta)$ for which $\frac{l(l+N-2)}{r^2} - \mu_{l,1} \leq 0$, since $v$ is increasing. That is, \eqref{eq:long} is satisfied for such $r$.
	For the remaining values of $r$, the inequality \eqref{eq:long} is trivial since the left-hand side of \eqref{eq:long} is positive, while the right-hand side is negative.
\end{proof}

Prior to the proof of Lemma \ref{lem:n+2}, we need to establish several auxiliary facts. 

\begin{lemma}\label{lem:2ev}
	Let $v$ be an  eigenfunction corresponding to the eigenvalue $\mu_{0,2}$  of the SL problem \eqref{eq:ode}, \eqref{eq:bcN}. Then $v'$ does not vanish on $(\al, \be)$.
\end{lemma}
\begin{proof}
	We know from the Sturm-Liouville theory that $v$ vanishes exactly once in $(\al,\be).$ 
	Let  $\ga\in (\al,\be)$ be such that $v(\ga)=0$.
	Assume, without loss of generality, that $v>0$ in $(\al,\ga)$ and $v<0$ in $(\ga,\be)$. 
	Taking any $s\in (\al,\ga]$ and integrating \eqref{eq:ode3} 
	(with $l=0$) 
	from $\al$ to $s$, we  obtain
	$$
	-s^{N-1}v'(s)=\mu_{0,2}\int_\al^s v(r)r^{N-1} \, dr>0.
	$$
	Similarly, for any $s\in [\ga,\be)$, we get 
	$$
	s^{N-1}v'(s)=\mu_{0,2}\int_s^\be v(r)r^{N-1} \, dr<0.
	$$
	Thus, $v'(s)<0$ for all $s\in (\al,\be)$. 
\end{proof}

Let us now consider the Sturm-Liouville eigenvalue problem \eqref{eq:ode}  
with the Dirichlet boundary conditions
\begin{equation}\label{eq:bcD}
u(\alpha)=0 \quad \text{and} \quad u(\beta)=0.
\end{equation}
The corresponding eigenvalues $\la_{l,j}$ have the following Courant-Fischer variational characterisation:
\begin{equation}\label{eq:ladef}
\la_{l,j}=\min_{X\in \mathcal{X}_j}\max_{u\in X\setminus \{0\}}R_l(u),
\quad l\in \mathbb{N}_0, ~j \in \mathbb{N},
\end{equation}
where $\mathcal{X}_j$ is the collection of all $j$-dimensional subspaces of $H^1_0((\alpha,\beta);r^{N-1})$, and $R_l$ is defined as in \eqref{eq:RQ}.
Comparing the variational characterizations \eqref{eq:mudef} and \eqref{eq:ladef}, it can be shown that 
\begin{equation}\label{eq:mutaulambda}
\mu_{l,j}<  \la_{l,j} 
\quad \text{for all}~ l\in \mathbb{N}_0,~ j \in \mathbb{N}. 
\end{equation}

\begin{proposition}\label{prop:mu2tau2}
	We have 
	$\mu_{0,2}=\la_{1,1}$, and hence $\mu_{1,1}<\mu_{0,2}$.
\end{proposition}
\begin{proof}
	Let $v$ be an eigenfunction of the SL problem \eqref{eq:ode}, \eqref{eq:bcN} corresponding to $\mu_{0,2}$. 
	Differentiating \eqref{eq:ode} with $l=0$, we deduce that $u=v'$ satisfies the
	equation \eqref{eq:ode} with $l=1$ under the Dirichlet boundary conditions \eqref{eq:bcD}, namely,
	\begin{equation}\label{eq:SLD}
	-u''(r)-\frac{N-1}{r}u'+\frac{N-1}{r^2}u
	=
	\mu_{0,2}u \quad \text{in}~ (\al,\be),
	\quad u(\al)=u(\be)=0.
	\end{equation}
	By Lemma \ref{lem:2ev}, $v'$ does not vanish on $(\al,\be)$, and hence $u$ must be the first eigenfunction of \eqref{eq:SLD}. 
	Therefore, we conclude that $\mu_{0,2}=\la_{1,1}$, and \eqref{eq:mutaulambda} yields $\mu_{1,1}<\mu_{0,2}$.
\end{proof}

Now we are ready to prove Lemma \ref{lem:n+2}.
\begin{proof}[Proof of Lemma \ref{lem:n+2}]
	Since the inequality $\mu_{1,1} < \mu_{2,1}$ is given by \eqref{eqn:col}, we only need to show that $\mu_{2,1}<\mu_{0,2}$. 	
	Assume, without loss of generality, that $\beta=1$, and let us write $\mu_{l,k} = \mu_{l,k}(\alpha)$ and $\lambda_{l,k} = \lambda_{l,k}(\alpha)$ to stress the dependence on $\alpha$.
	
	First, we consider the case $\alpha=0$, i.e., the case of the ball. 
	It is known that $\mu_{2,1}(0) = \left(p_{\frac{N}{2},1}^{(2)}\right)^2$ and $\mu_{0,2}(0) = \left(p_{\frac{N}{2},2}^{(0)}\right)^2$, see Section \ref{sec:spectrum}.  
	Moreover, we have $p_{\nu,k}^{(0)} = j_{\nu,k}$, 
	where $j_{\nu,k}$ is the $k$-th positive zero of the Bessel function $J_\nu$,
	see, e.g., \cite[p.~549]{LS}.
	Therefore, applying the estimates 
	$$
	p_{\frac{N}{2},1}^{(2)} < \sqrt{2N+8}
	\quad \text{and} \quad
	j_{\frac{N}{2},2} > \frac{N}{2} + \frac{3\pi+1}{2}
	$$
	from \cite[(1)]{LS} and \cite[(4.6)]{IS}, respectively, it is not hard to deduce that the corresponding bounds are strictly ordered, 
	which yields the desired inequality $\mu_{2,1}(0)<\mu_{0,2}(0)$. 
	
	Second, we assume $\alpha>0$.
	Notice that $\mu_{2,1}(\alpha)$ is nonincreasing with respect to $\alpha$.
	Indeed, denoting by $v$ the first positive eigenfunction of the SL problem \eqref{eq:ode}, \eqref{eq:bcN} associated to $\mu_{2,1}(\alpha)$, we have
	\begin{equation}\label{eq:muprime-repers}
	\frac{\partial \mu_{2,1}(\alpha)}{\partial \alpha} 
	= 
	-v^2(\alpha) (2N \alpha^{N-3} - \mu_{2,1}(\alpha)\alpha^{N-1}) 
	= 
	-v(\alpha) v''(\alpha)\alpha^{N-1} \leq 0,
	\end{equation}
	where the first equality is given by, e.g., \cite[Theorem 4.1, 1]{KZ}, and the second equality follows from \eqref{eq:ode} and the fact that $v'(\alpha)=0$. 
	The inequality in \eqref{eq:muprime-repers} is implied by $v''(\alpha) \geq 0$ which, in its turn, is a consequence of $v(\alpha) \geq 0$, $v'(\alpha)=0$, and $v'>0$ on $(\alpha,\beta)$ by Lemma \ref{lem:v}.	
	On the other hand, in view of the domain monotonicity of the Dirichlet eigenvalues, $\lambda_{1,1}(\alpha)$ is increasing with respect to $\alpha$, and hence the same holds for $\mu_{0,2}(\alpha)$ by Proposition \ref{prop:mu2tau2}.
	Therefore, in view of the convergence \eqref{eq:mulj-conv}, we conclude that
	\begin{equation*}\label{eq:lower2}
	\mu_{2,1}(\alpha) \leq \mu_{2,1}(0) < \mu_{0,2}(0) <
	\mu_{0,2}(\alpha) \quad \text{for any}~ \alpha \in (0,1),
	\end{equation*}
	which completes the proof.
\end{proof}

\section{}\label{sec:appendix}

In this section, we provide several integral equalities and orthogonality results which will be needed in the proof of Theorem \ref{thm1} to show the orthogonality of elements of the set $X_k$ in $L^2(\Omega)$ and $H^1(\Omega)$.

\subsection{Auxiliary integral equalities}\label{sec:Integrals}
We start by proving several integral equalities for domains with symmetries imposed in Theorem \ref{thm1}. 
\begin{proposition}\label{2Orthogonality00}
  Let $\Om\subset\mathbb{R}^2$ be a bounded domain.
  Let $g$ be a positive radial function on $\mathbb{R}^2$. 
  Then for $i,j \in\{1,2\}$ with $i\ne j$ and any $m \in \mathbb{N}_0$ the following assertions hold:
  \begin{enumerate}[label={\rm(\roman*)}]
      \item\label{2Orthogonality00:1} If $\Omega$ is centrally symmetric (or, equivalently, symmetric of order $2$), then
      \begin{equation}\label{eq:2proport1}
      \int_\Omega g(r)x_ix_j^{2m}\,dx=0
      \quad \text{and} \quad 
      \int_{\Omega} g(r) x_i^{2m+1}\,dx = 0.
      \end{equation}
      \item\label{2Orthogonality00:2} If $\Omega$ is symmetric of order $4$, then
      \begin{align}
      \label{eq:2proport0}
      \int_\Omega g(r)x_ix_j\,dx=0
      \quad \text{and} \quad 
      \int_{\Omega} g(r)x_i^{2m}\,dx
      =
      \int_{\Omega} g(r)x_j^{2m}\,dx.
      \end{align}
      Moreover, there exists $T\in SO(2)$ such that 
      \begin{equation}\label{eq:2proport3}
      \int_{T(\Omega)} g(r)x_ix_j^{3}\,dx
      =
      0.
      \end{equation}
      \item\label{2Orthogonality00:3} If $\Omega$ is symmetric of order $8$, then we have  
      \begin{align}\label{eq:2proport2}
       \int_\Omega g(r)x_ix_j^{3}\,dx=0
       \quad \text{and} \quad 
     \int_\Omega g(r) \left(x_i^2-x_j^2\right)^2 \,dx
	=
	4\int_\Omega g(r) x_i^2 x_j^2\,dx.
      \end{align}
  \end{enumerate}
  \end{proposition}
  
  \begin{proof}
  \ref{2Orthogonality00:1} 
 By the central symmetry of $\Omega$,  the transformation $x=-y$ (or, equivalently, $x=R^{2\pi/2}_{1,2}(y)$) yields
  $$
  \int_{\Omega} g(r)x_i x_j^{2m}\,dx=-\int_{\Omega} g(r)y_i y_j^{2m}\,dy
  \quad 
  \text{and}
  \quad 
  \int_{\Omega} g(r) x_i^{2m+1}\,dx = -\int_{\Omega} g(r) y_i^{2m+1}\,dy,
  $$
  which implies \eqref{eq:2proport1}.
  
  \ref{2Orthogonality00:2} By the symmetry of order 4, the transformation $x=R^{2\pi/4}_{1,2}(y)$ (see \eqref{eq:R4}) yields
  \begin{equation}\label{eqn:2d1}
    \int_{\Omega} g(r)x_i x_j^{2m+1}\,dx=-\int_{\Omega} g(r)y_j y_i^{2m+1}\,dy
    \quad \text{and} \quad  
    \int_{\Omega} g(r)x_i^{2m}\,dx=\int_{\Omega} g(r)y_j^{2m} \,dy,
  \end{equation} 
  and hence \eqref{eq:2proport0} follows. 
  
 Consider now the map $I(\theta)= \int_{R_{1,2}^{\theta}(\Om)}g(r)x_i x_j^{3}\,dx$ for $0\le \theta \le \pi/4$.
Using the transformation $x=R^{2\pi/8}_{1,2}(y)$ (see \eqref{eq:R8}) and applying \eqref{eqn:2d1}, we obtain
\begin{equation}\label{eq:rot}
    \begin{aligned}
    I(\pi/4)
    &=
    \int_{R^{2\pi/8}_{1,2}(\Omega)} g(r)x_i x_j^{3} \,dx
    =
    \frac{1}{4}\int_{\Omega} g(r)(y_i- y_j) (y_i+y_j)^{3}\,dy\\
    &=
    \frac{1}{4}\int_{\Omega} g(r) (y_i^4+2y_i^3y_j-2y_iy_j^3-y_j^4)\,dy=-\int_{\Omega} g(r)y_i y_j^3 \,dy=-I(0).
\end{aligned}
\end{equation}
  Since $I$ is continuous, we conclude that either $I(0)=0$, or $I(\theta)=0$ for some $0< \theta <\pi/4$, which establishes \eqref{eq:2proport3}. 

  \ref{2Orthogonality00:3} Since $\Om$ is symmetric of order 8, we have $R^{2\pi/8}_{1,2}(\Om)=\Om$. 
  Thus, arguing as in the assertion \ref{2Orthogonality00:2} above, we deduce from \eqref{eq:rot} that $I(0)=-I(0)$, which yields $\int_{\Omega} g(r)x_i x_j^3 \,dx=0$. 
  Furthermore, by applying the transformation $x = R^{2\pi/8}_{1,2}(y)$ (see \eqref{eq:R8}) we easily derive the second equality in \eqref{eq:2proport2}. 
\end{proof}

\begin{remark}\label{rem:lem:N2}
	The presence of rotation $T$ in Proposition \ref{2Orthogonality00} \ref{2Orthogonality00:2} is not avoidable, in general, as a simple example when $\Omega$ is a square and $g=1$ shows. 
	However, clearly, $T(\Omega)$ is also symmetric of order $4$ and $\mu_k(T(\Omega)) = \mu_k(\Omega)$ for any $k \in \mathbb{N}$.
	Thus, in applications of Proposition \ref{2Orthogonality00} \ref{2Orthogonality00:2} we will often assume, without loss of generality, that $T$ is the identity, i.e., that \eqref{eq:2proport3} holds for $\Omega$ itself.
\end{remark}

\begin{proposition}\label{Orthogonality00}
  Let $\Om\subset\mathbb{R}^N$ be a bounded domain, $N \geq 3$. 
  Let $g$ be a positive radial function on $\mathbb{R}^N$. 
  Then for any $i,j\in \{1,2,\ldots, N\}$ with $i\neq j$ and any $m\in \mathbb{N}_0$ the following assertions hold:
  \begin{enumerate}[label={\rm(\roman*)}]
      \item\label{Orthogonality00:1} If $\Omega$ is centrally symmetric, then
      \begin{equation*}\label{eq:proport1}
      \int_\Omega g(r)x_ix_j^{2m}\,dx=0
      \quad \text{and} \quad 
      \int_{\Omega} g(r) x_i^{2m+1}\,dx = 0.
      \end{equation*}
       \item\label{Orthogonality00:2} If $\Omega$ is symmetric of order $2$, then 
      \begin{equation}\label{eq:3proport1}
      \int_\Omega g(r)x_ix_j^{m}\,dx=0
      \quad \text{and} \quad 
      \int_{\Omega} g(r) x_i^{2m+1}\,dx = 0.
      \end{equation}
        \end{enumerate}
  \end{proposition}
\begin{proof}

  \ref{Orthogonality00:1} The proof follows exactly as in Proposition \ref{2Orthogonality00} \ref{2Orthogonality00:1}.

 \ref{Orthogonality00:2}
 Since $N\ge 3,$ we can choose $k$ so that $k\neq i,j$. 
 Using the transformation $x=R^{2\pi/2}_{i,k}(y)$, we obtain
 $$
 \int_{\Omega} g(r)x_i x_j^{m}\,dx=-\int_{\Omega} g(r)y_i y_j^{m}\,dy,
 $$
 and hence the first equality in \eqref{eq:3proport1} follows.
 The second equality in \eqref{eq:3proport1} can be shown as in Proposition \ref{2Orthogonality00} \ref{2Orthogonality00:1}.
 \end{proof}

\begin{proposition}\label{Orthogonality0}
  Let $\Om\subset\mathbb{R}^N$ be a bounded domain symmetric of order $4$, $N \geq 3$.
  Let $g$ be a positive radial function on $\mathbb{R}^N$.
  Then the following assertions hold:
\begin{enumerate}[label={\rm(\roman*)}]
\item\label{Orthogonality0:2} 
For any $i,j,k,l\in \{1,2,\ldots, N\}$ with $i\ne j$ and $k,l \not\in \{i,j\}$, and any $m,n\in \mathbb{N}_0$, we have
$$
\int_\Omega g(r) x_i x_j x_k^{m} x_l^n \,dx=0.
$$
\item\label{Orthogonality0:3} There exist constants $A,B > 0$ such that 
$$
\int_\Omega g(r)x_i^2 \,dx=A
~\text{ and }~
\int_\Omega g(r)x_i^4 \,dx=B 
~\text{ for any }~
i \in \{1,2,\dots,N\}.
$$
\item\label{Orthogonality0:4} 
There exists a constant $C > 0$ such that  
$$
\int_\Omega g(r)x_i^2 x_j^2\,dx=C
~\text{ for any }~
i,j \in \{1,2,\dots,N\}
~\text{ with }~ 
i \neq j.
$$
\end{enumerate}
\end{proposition}
\begin{proof}
\ref{Orthogonality0:2} 
Since $i \neq j$ and $k,l \not\in \{i,j\}$, we use the transformation  $x=R^{2\pi/4}_{i,j}(y)$ (see \eqref{eq:R4}) to obtain
$$
\int_{\Omega} g(r)x_i x_j x_k^m x_l^n \,dx
=
-\int_{\Omega} g(r)y_iy_j y_k^{m}y_l^n \,dy,
$$
and hence the result follows.

\ref{Orthogonality0:3} 
Fixing any $i\ne 1$ and taking $x=R^{2\pi/4}_{1,i}(y)$, we obtain
$$
\int_\Omega g(r)x_i^2 \,dx
= 
\int_\Omega g(r)y_1^2 \,dy
\quad
\text{and}
\quad
\int_\Omega g(r)x_i^4 \,dx
= 
\int_\Omega g(r)y_1^4 \,dy,
$$
which yields the existence of the required constants $A,B>0$.

\ref{Orthogonality0:4} 
Assume first that
$k,l \in \{1,2,\dots,N\}$ are such that $k \neq l$ and $\{k,l\} \neq \{i,j\}$, but $\{k,l\} \cap \{i,j\} \neq \emptyset$. 
Without loss of generality, we assume $k=i$ and $l \neq j$.
Then, applying the transformation  $x=R^{2\pi/4}_{j,l}(y)$, we obtain
$$
\int_\Omega g(r)x_i^2 x_j^2\,dx=\int_\Omega g(r)y_k^2 y_l^2 \,dy.
$$
Assume now that $k,l$ are such that $k \neq l$ and $\{k,l\} \cap \{i,j\} = \emptyset$.
Then, applying the transformation  $x=R^{2\pi/4}_{i,k}\big(R^{2\pi/4}_{j,l}(y)\big)$, we get
$$
\int_\Omega g(r)x_i^2 x_j^2\,dx=\int_\Omega g(r)y_k^2 y_l^2 \,dy.
$$
The combination of these two cases gives the existence of the required constant $C>0$.
\end{proof}

Finally, we provide several algebraic identities that we use in the proof of Theorem \ref{thm1}. 
\begin{lemma}\label{lem:algebraic}
	For any $x=(x_1,x_2,\ldots, x_N)\in \mathbb{R}^N$ we have the following identities:
	\begin{align}
		&2 \sum_{i=1}^{N-1} \sum_{j=i+1}^N  x_i^2 x_j^2 
		+
		\sum_{i=1}^{N-1} \frac{1}{i(i+1)} \left(\sum_{j=1}^i x_j^2 - i x_{i+1}^2\right)^2
		= 
		\frac{N-1}{N}r^4, \label{eq:alg4}\\
		&\sum_{i=1}^{N-1}\frac{1}{i(i+1)} \left(\sum_{j=1}^i x_j^2 + i^2 x_{i+1}^2\right)= \frac{N-1}{N} r^2,\label{eq:alg12}\\
		&\sum_{i=1}^{N-1} \sum_{j=i+1}^N (x_i^2+x_j^2)= (N-1)r^2. \label{eq:alg2}
	\end{align}
\end{lemma}
\begin{proof}
	The proof follows easily by the induction with respect to the dimension $N$.
\end{proof}

\subsection{Bases of $H_1$ and $H_2$}\label{sec:Orthogonal}
In this section, we provide an auxiliary information on the sets of homogeneous harmonic polynomials $H_1$ and $H_2$.
Let $\gamma$ be a multi-index, i.e, 
$$
\gamma=(\gamma_1,\gamma_2,\ldots, \gamma_N),~  \gamma_i\in \mathbb{N}_0.
$$ 
For such $\gamma$ we use the following standard notation:
\begin{align*}
	|\gamma|= \gamma_1+\ldots +\gamma_N,
	\quad
	\gamma! = \gamma_1!\ldots \gamma_N!,
	\quad
	x^\gamma =x_1^{\gamma_1}\ldots x_N^{\gamma_N}.
\end{align*}
For $p=\displaystyle\sum_{|\gamma|= k}a_{\gamma} x^{\gamma}\in H_k$ and $q=\displaystyle\sum_{|\gamma|= l}b_{\gamma} x^{\gamma}\in H_l$, we define
\begin{equation}\label{ip}
	\inpr{p,q}=
	\left\{
	\begin{aligned}
		&\frac{1}{l!}\sum_{|\gamma|=l}\gamma!a_\gamma b_\gamma &&\text{ if } k=l,\\
		&0  &&\text{ if } k\ne l.
	\end{aligned}
	\right.
\end{equation} 
In particular, for any monomials $x^\gamma,x^\beta \in H_l$ we have 
\begin{equation}\label{ip:mon}
\langle x^\gamma,x^\be \rangle=
\left\{
\begin{aligned}
	&\frac{\gamma!}{l!} &&\text{ if } \gamma=\be,\\
	&0 &&\text{ if } \gamma\neq \be.
\end{aligned}
\right.
\end{equation}
It is not hard to verify that $\left<\cdot,\cdot\right>$ is an inner product on $\cup_{l=0}^\infty H_l$,
see, e.g., \cite[Chapter 5]{axler}.

Let us now determine orthogonal bases for $H_1$ and $H_2$.
Consider the following sets of homogeneous harmonic polynomials:
\begin{align*}
	Z_1&:= 
	\left\{
	x_i:~ 
	i = 1,2,\ldots, N 
	\right\},\\
	Z_2&:= 
	\left\{
	x_i x_j:~ 
	i<j \text{ and } i,j = 1,2,\ldots, N
	\right\},\\
	Z_3&:= 
	\left\{
	x_i^2- x_{i+1}^2:~ 
	i = 1,2,\ldots, N-1 
	\right\}.
\end{align*}
Clearly, $Z_1 \subset H_1$ and $Z_2,Z_3 \subset H_2$. 
Moreover, we deduce from \eqref{ip} and \eqref{ip:mon} that
\begin{equation}\label{eq:<p,p>}
	\inpr{p,p}
	=
	\left\{
	\begin{aligned}
		&1 &&\text{for }~ p\in Z_1,\\
		&\frac{1}{2} &&\text{for }~  p\in Z_2, \\
		&2 &&\text{for }~  p\in Z_3,
	\end{aligned}
	\right. 
\end{equation}
and if $p\ne q$, then 
\begin{equation}\label{eq:intpq}
	\inpr{p,q}
	=
	\left\{
	\begin{aligned}
		&-1 &&\text{if }  p,q \in Z_3 \text{ have a common index},\\
		&0 &&\text{otherwise}.
	\end{aligned}
	\right.
\end{equation}

Combining the mutual orthogonality of elements of $Z_1$ given by \eqref{eq:intpq} with the fact that $\# (Z_1) = \text{dim}\, H_1 = N$ (see Proposition \ref{prop:multiplicity}), we deduce that $Z_1$ is an orthonormal basis of $H_1$.
Since $\# (Z_2 \cup Z_3) = \text{dim}\, H_2 = \frac{(N+2)(N-1)}{2}$, it is not hard to see that $Z_2 \cup Z_3$ is a basis of $H_2$.
However, the elements of $Z_3$ are not always orthogonal with each other, see \eqref{eq:intpq}.
Applying the Gram-Schmidt orthogonalization procedure to $Z_3$ with respect to the inner product \eqref{ip}, we obtain the following subset of $H_2$: 
$$
\widetilde{Z_3}
:=
\left\{
\frac{1}{\sqrt{i(i+1)}}\left(\sum_{j=1}^ix_j^2 - i x_{i+1}^2\right):~ 
i =1,2,\ldots,N-1
\right\}.
$$
Since the elements of $Z_2$ were orthogonal to the elements of $Z_3$, we conclude that $Z_2 \cup \widetilde{Z_3}$ is an orthogonal basis of $H_2$.
In particular, we have
\begin{equation*}\label{eq:orthogon}
	\left<p,q\right>=0
	~\text{ for}~ p,q \in Z_1 \cup Z_2 \cup \widetilde{Z_3} \text{ with } p \neq q,\quad
	\inpr{p,p}=1 ~\text{ for } p\in \widetilde{Z_3}.
\end{equation*}

\subsection{Orthogonality}\label{sec:Orthogonality}
In this section, we show several orthogonality results which will be used to obtain the mutual orthogonality of elements of the set $X_k$ with respect to the scalar products in $L^2(\Omega)$ and $H^1(\Omega)$, where $X_k$ are defined in the proof of Theorem \ref{thm1} via the elements of $Z_1$, $Z_2$, and $\widetilde{Z_3}$.

In order to deal with the orthogonality in $H^1(\Omega)$, let us provide several useful expressions.
For a radial $C^1$-function $g$ and for any $p\in Z_i$, $i=1,2,3$, we have
$$
\nabla(g(r)p)=g'(r)\frac{x}{r}p+g(r)\nabla p
$$
and
\begin{equation*}\label{eq:nabla1}
	\nabla(g(r)p) \nabla(g(r)q) =
	(g'(r))^2pq+g'(r)g(r)\frac{x}{r} (p\nabla q+q\nabla p)+g^2(r)\nabla p \nabla q.
\end{equation*}   
It is easy to see that
$$
\nabla p
=
\left\{
\begin{aligned}
	&e_i &&\text{for } p=x_i,\\
	&x_je_i+x_ie_j &&\text{for } p=x_ix_j,\\
	&2(x_ie_i-x_je_j) &&\text{for } p=x_i^2-x_j^2,
\end{aligned}
\right.
\qquad 
x  \nabla p
=
\left\{
\begin{aligned}
	&p &&\text{for } p=x_i,\\
	&2p &&\text{for } p=x_ix_j,\\ 
	&2p &&\text{for } p=x_i^2-x_j^2,
\end{aligned}
\right.
$$ 
and 
\begin{equation} \label{eq:nabla2}
	x (p\nabla q+ q \nabla p)
	=
	\left\{
	\begin{aligned}
		&2pq &&\text{for } p,q\in Z_1, \\
		&4pq &&\text{for } p,q\in Z_2\cup Z_3,\\
		&3pq &&\text{for } p\in Z_1,q\in Z_2\cup Z_3.
	\end{aligned}
	\right.
\end{equation}
Furthermore, denoting by $\delta_{i,j}$ the Kronecker delta, we have
\begin{equation}\label{eq:nabla3}
	\nabla p   \nabla q
	=
	\left\{
	\begin{aligned}
		&\delta_{i,j} 
		&&\text{for } p=x_i,q=x_j,\\
		&x_l\delta_{i,k}+x_k\delta_{i,l} 
		&&\text{for } p=x_i,q=x_kx_l,\\
		&2(x_k\delta_{i,k}-x_l\delta_{i,l}) 
		&&\text{for } p=x_i,q=x_k^2-x_l^2,\\
		&x_i(x_k\delta_{j,l}+x_l \delta_{j,k})+  x_j(x_k\delta_{i,l}+x_l \delta_{i,k}) 
		&&\text{for } p=x_ix_j, q=x_kx_l,\\
		&2x_i(x_k\delta_{j,k}-x_l \delta_{j,l})+  2x_j(x_k\delta_{i,k}-x_l \delta_{i,l}) 
		&&\text{for } p=x_ix_j, q=x_k^2-x_l^2,\\
		&4x_i(x_k \delta_{i,k} - x_l \delta_{i,l})-4x_j(x_k \delta_{j,k} - x_l \delta_{j,l}) 
		&&\text{for } p=x_i^2-x_j^2, q=x_k^2-x_l^2.
	\end{aligned}
	\right.  
\end{equation}

\medskip
We will separately consider the cases $N=2$ and $N \geq 3$.
In the planar case $N=2$, $Z_3$ is equivalent to $\widetilde{Z_3}$, up to a multiplication by a constant.
Therefore, using Proposition \ref{2Orthogonality00} \ref{2Orthogonality00:1}, \ref{2Orthogonality00:2}, and the expressions above, we deduce the following result.
\begin{lemma}\label{lem:z1z2z3-tildeN2}
	Let $\Om\subset\mathbb{R}^2$ be a bounded domain symmetric of order $4$. 
	Let $g$ be a positive radial $C^1$-function on $\mathbb{R}^N$. 
	Then there exists $T \in SO(2)$ such that for any $p,q\in Z_1\cup  Z_2\cup \widetilde{Z_3}$ with $p\ne q$ we have 
	\begin{equation}\label{eq:orth0}
		\int_{T(\Om)}
		g(r)p(x)q(x)\,dx=0
		\quad \text{ and } \quad 
		\int_{T(\Om)} \nabla(g(r)p(x))  \nabla(g(r)q(x)) \,dx=0.
	\end{equation}
\end{lemma}
\begin{proof}
	By straightforward calculations.
\end{proof}

\begin{remark}\label{rem:lem:z1}
	The rotation $T$ in Lemma \ref{lem:z1z2z3-tildeN2} is used only in the case $p \in Z_2$, $q \in \widetilde{Z_3}$, that is, $p=x_1x_2$, $q=(x_1^2-x_2^2)/\sqrt{2}$, 
	while in all other cases \eqref{eq:orth0} holds for $\Omega$ itself.
\end{remark}

Let us now consider the case $N \geq 3$.
In order to avoid bulky calculations caused by the structure of $\widetilde{Z_3}$, we first give explicit relations between the integrals $\int_\Om g(r)  p(x)q(x)\,dx$, $\int_\Om \nabla(g(r)p(x))  \nabla(g(r)q(x)) \,dx$ and the inner product $\inpr{p,q}$ for polynomials $p,q \in Z_1 \cup Z_2 \cup Z_3$.
\begin{lemma}\label{Orthogonality2}
    Let $\Om\subset\mathbb{R}^N$ be a bounded domain symmetric of order $4$, $N \geq 3$. 
    Let $g$ be a positive radial function on $\mathbb{R}^N$. 
	Then 
\begin{equation}\label{eq:ortho1}
   \int_\Om g(r) p(x)q(x)\,dx=
   \left\{
   \begin{aligned}
   &A \inpr{p,q}  &&\text{for }~ p,q \in Z_1,\\
   &2C\left<p,q\right>  &&\text{for }~  p,q \in Z_2,\\
   &(B-C)\left<p,q\right> &&\text{for }~ p,q \in Z_3, \\
   &0 && \text{otherwise},
   \end{aligned}
  \right.
\end{equation}      
where $A$, $B$, and $C$ are given by Proposition \ref{Orthogonality0} \ref{Orthogonality0:3}, \ref{Orthogonality0:4}.
\end{lemma}
\begin{proof}
We see from the definitions of $A$, $B$, and $C$ that
\begin{equation}\label{eq:<p,p>2}
\int_\Om g(r) p^2(x)\,dx
=
\left\{
\begin{aligned}
        &A &&\text{for }~ p\in Z_1,\\
        &C &&\text{for }~  p\in Z_2, \\
        &2(B-C) &&\text{for }~    p\in Z_3.
\end{aligned}
\right.
\end{equation}
Moreover, if $p\ne q$, then, using Proposition \ref{Orthogonality00} \ref{Orthogonality00:2} and Proposition \ref{Orthogonality0}, we obtain
\begin{equation}\label{eq:intpq2}
\int_\Om g(r) p(x)q(x)\,dx
=
\left\{
\begin{aligned}
&-(B-C) &&\text{if } p,q \in Z_3 \text{ have a common index},\\
&0 &&\text{otherwise}.
\end{aligned}
\right.
\end{equation}
Combining the expressions \eqref{eq:<p,p>2}, \eqref{eq:intpq2} with \eqref{eq:<p,p>}, \eqref{eq:intpq}, we easily derive \eqref{eq:ortho1}.
\end{proof}

\begin{lemma}\label{OrthogonalityGradient}
    Let $\Om\subset\mathbb{R}^N$ be a bounded domain symmetric of order $4$, $N \geq 3$. 
    Let $g$ be a positive radial $C^1$-function on $\mathbb{R}^N$. 
    In view of Proposition \ref{Orthogonality0} \ref{Orthogonality0:3}, \ref{Orthogonality0:4}, for any $i,j \in \{1,2,\dots,N\}$ we denote 
    \begin{align*}
    A'= \int_\Om \left[g'(r)+\frac{2g(r)}{r}\right]g'(r)x_i^2 \,dx, 
    \quad
       B'=\int_\Om \left[g'(r)+\frac{4g(r)}{r}\right]g'(r)x_i^4 \,dx,\\
    C'=\int_\Om \left[g'(r)+\frac{4g(r)}{r}\right]g'(r)x_i^2x_j^2 \,dx,    \quad 
    D'=\int_\Om g^2(r) x_i^2\,dx, 
    \quad 
    E'=\int_\Om g^2(r)\,dx.
    \end{align*}
Then
\begin{equation}\label{eqn:grad3}
\int_\Om \nabla(g(r)p(x))  \nabla(g(r)q(x)) \,dx
=
\left\{
\begin{aligned}
&(A'+E')\inpr{p,q} &&\text{for } p,q\in Z_1,\\
&(2C'+4D')\inpr{p,q} &&\text{for } p,q\in Z_2,\\
&(B'-C'+4D')\inpr{p,q} &&\text{for }  p,q \in Z_3,\\
&0 &&\text{otherwise}.
\end{aligned}
\right.
\end{equation}
\end{lemma}
\begin{proof}

Let us take any $p,q\in Z_1 \cup Z_2 \cup Z_3$ such that $p\ne q$. 
If $p$ and $q$ have no common indices or have two common indices, then we easily get from \eqref{eq:nabla3} that  $\nabla p   \nabla q=0$. 
If $p\in Z_1\cup Z_2$, $q\in Z_2 \cup Z_3$, and $p,q$ have exactly one common index, then using Proposition \ref{Orthogonality00} \ref{Orthogonality00:2} (with $g^2(r)$ instead of $g(r)$), we derive from \eqref{eq:nabla3} that  $\int_\Om g^2(r) \nabla p   \nabla q \,dx=0$.  
Thus, we are left with the case when $p,q\in Z_3$ and they have exactly one common index. 
Assume, without loss of generality, that 
$p=x_i^2-x_{i+1}^2$ and $q=x_{i+1}^2-x_{i+2}^2$. 
Then $\inpr{p,q}=-1$ (see \eqref{eq:intpq}) and
$$
\int_\Om g^2(r) \nabla p   \nabla q \,dx=-4\int_\Om g^2(r)x_{i+1}^2\,dx=4D'\inpr{p,q}. 
$$ 
Therefore, for $p,q\in Z_1 \cup Z_2 \cup Z_3$ with $p\ne q$, we have
\begin{equation}\label{eqn:grad0}
\int_\Om g^2(r) \nabla p   \nabla q \,dx
=
\left\{
\begin{aligned}
&4D'\inpr{p,q} &&\text{for } p,q\in Z_3,\\
&0 &&\text{otherwise}.
\end{aligned}
\right.
\end{equation}   
On the other hand, recalling \eqref{eq:<p,p>}, in the case $p=q$ we get
\begin{equation}\label{eqn:grad1}
\int_\Om g^2(r) |\nabla p|^2\,dx
=
\left\{
\begin{aligned}
&E'\inpr{p,p} 
&&\text{for } p\in Z_1, \\
&4D'\inpr{p,p} 
&&\text{for } p\in Z_2 \cup Z_3.
\end{aligned}
\right.
\end{equation}                        
Finally, for $p,q\in Z_1 \cup Z_2 \cup Z_3$,
using Proposition \ref{Orthogonality00} \ref{Orthogonality00:2}, Proposition \ref{Orthogonality0}, and \eqref{eq:nabla2}, we obtain
\begin{equation}\label{eqn:grad2}
\int_\Om \left[(g'(r))^2pq+g'(r)g(r)\frac{x}{r} (p\nabla q+q\nabla p)\right] dx
=
\left\{
\begin{aligned} 
&A' \inpr{p,q} &&\text{for } p,q\in Z_1,\\
&2C'\inpr{p,q} &&\text{for } p,q\in Z_2,\\
&(B'-C')\inpr{p,q} &&\text{for } p,q\in Z_3,\\
&0 &&\text{otherwise}.
\end{aligned}
\right.
\end{equation}
Combining \eqref{eqn:grad0}, \eqref{eqn:grad1}, and \eqref{eqn:grad2}, we easily derive \eqref{eqn:grad3}.
\end{proof}

Now we are ready to obtain a counterpart of Lemma \ref{lem:z1z2z3-tildeN2} in the case $N \geq 3$.
\begin{lemma}\label{lem:z1z2z3-tilde}
Let $\Om\subset\mathbb{R}^N$ be a bounded domain symmetric of order $4$, $N \geq 3$. 
Let $g$ be a positive radial $C^1$-function on $\mathbb{R}^N$. 
Then for any $p,q\in Z_1\cup  Z_2\cup \widetilde{Z_3}$ with $p\ne q$ we have 
\begin{align*}
\int_\Om 
g(r)p(x)q(x)\,dx=0
\quad \text{ and } \quad 
\int_\Om \nabla(g(r)p(x))  \nabla(g(r)q(x)) \,dx=0.
\end{align*}
\end{lemma}
\begin{proof}
For $p,q\in Z_3$, we get from \eqref{eq:ortho1} that
$$
(B-C)
\left<p,q\right>
=
\int_\Om g(r) p(x)q(x)\,dx.
$$
Taking $p=q$ and recalling that $g$ is positive, we see that $B>C$.
That is, the integral $\frac{1}{B-C}\int_\Om g(r) p(x)q(x)\,dx$ defines an equivalent inner product on $Z_3$, and hence the Gram-Schmidt orthogonalisation of $Z_3$ with respect to this inner product also produces $\widetilde{Z_3}$.
Thus, using \eqref{eq:ortho1}, we easily conclude that 
$$
\int_\Om g(r)p(x)q(x)\,dx=0
\quad \text{ for } p,q\in Z_1\cup Z_2\cup \widetilde{Z_3} \text{ with } p \neq q.
$$
Similarly, using \eqref{eqn:grad3}, we conclude that 
$$
\int_\Om \nabla(g(r)p(x))  \nabla(g(r)q(x)) \,dx=0 
\quad \text{ for } p,q\in Z_1\cup Z_2\cup \widetilde{Z_3} \text{ with } p \neq q,
$$
which completes the proof.
\end{proof}

\begin{remark}\label{remark:norm}
For any $p \in Z_1 \cup Z_2 \cup Z_3$, from \eqref{eqn:grad3} we get
\begin{align*} 
&\int_\Om |\nabla g(r)p(x)|^2\,dx\\
&=
\left\{
\begin{aligned}
&\int_\Om \left[{(g'(r))^2}p^2(x)+\frac{2g'(r)g(r)}{r}p^2(x)+g^2(r)\right] dx  &&\text{for }~ p=x_i, \\
&\int_\Om \left[{(g'(r))^2}p^2(x)+\frac{4g'(r)g(r)}{r}p^2(x)+g^2(r)(x_i^2+x_j^2)\right] dx  &&\text{for }~ p=x_i x_j, \\
&\int_\Om \left[{(g'(r))^2}p^2(x)+\frac{4g'(r)g(r)}{r}p^2(x)+4g^2(r)(x_i^2+x_j^2))\right] dx  &&\text{for }~ p=x_i^2-x_j^2.
\end{aligned}
\right.
\end{align*}
In particular, for the function $G:=G_l$ defined as in \eqref{eq:G} and for every $i\in \{1,2,\ldots,N\}$ we have
\begin{align*}
&\int_\Om \left|\nabla\left( \frac{G(r)}{r}x_i\right)\right|^2dx
=
\int_\Om \left(
\frac{(G'(r))^2}{r^2}x_i^2
-
\frac{G^2(r)}{r^4}x_i^2
+
\frac{G^2(r)}{r^2}
\right) dx,\\
&\int_\Om \left|\nabla\left( \frac{G(r)}{r^2}x_ix_j\right)\right|^2dx
=
\int_\Om \left(
\left[
\frac{(G'(r))^2}{r^4}
-
\frac{4G^2(r)}{r^6}
\right]
x_i^2x_j^2
+
\frac{G^2(r)}{r^4}(x_i^2+x_j^2)\right) dx.
\end{align*}
Further, we can deduce that
\begin{align*}
&\int_\Om \left|\nabla\left( \frac{G(r)}{r^2\sqrt{i(i+1)}}\left[\sum_{j=1}^ix_j^2-ix_{i+1}^2
\right]\right)\right|^2 dx\\
&=\frac{1}{i(i+1)}\int_\Om\left( \left[\frac{(G'(r))^2}{r^4}-\frac{4G^2(r)}{r^6}\right]\left[\sum_{j=1}^ix_j^2-ix_{i+1}^2
\right]^2+ \frac{4G^2(r)}{r^4}\left(\sum_{j=1}^ix_j^2+ix_{i+1}^2
\right)\right)dx.
\end{align*}
\end{remark}

	\medskip
	\noindent
	\textbf{Acknowledgments.}
	T.V.~Anoop was supported by the INSPIRE Research Grant\\ DST/INSPIRE/04/2014/001865.
	V.~Bobkov was supported in the framework of implementation of the development program of Volga Region Mathematical Center (agreement no.~075-02-2021-1393).
	The authors would like to thank Professor L. Sadun for a helpful discussion regarding Section \ref{subsec:q}. 
	The authors would also like to thank the anonymous referees for valuable suggestions and remarks.


\begin{thebibliography}{99}	

	\bibitem{Abele}
	Abele, D., \& Kleefeld, A. (2020). New Numerical Results for the Optimization of Neumann Eigenvalues. In Computational and Analytic Methods in Science and Engineering (pp. 1-20). Birkh\"auser, Cham.
	\href{https://doi.org/10.1007/978-3-030-48186-5_1}{doi:10.1007/978-3-030-48186-5\_1} 


		\bibitem{AK}
		Anoop, T. V., \& Kumar, K. A. (2020). On reverse Faber-Krahn inequalities.
		Journal of Mathematical Analysis and Applications, 485(1), 123766.
		\href{https://doi.org/10.1016/j.jmaa.2019.123766}{doi:10.1016/j.jmaa.2019.123766}


    \bibitem{AFG}
    Antoneli, F., Forger, M., \& Gaviria, P. (2012). Maximal subgroups of compact Lie groups.
    Journal of Lie Theory, 22(4), 949-1024.
    \url{http://www.heldermann.de/JLT/JLT22/JLT224/jlt22043.htm}
    
    
    \bibitem{AO1}
    Antunes, P. R., \& Freitas, P. (2012). Numerical optimization of low eigenvalues of the Dirichlet and Neumann Laplacians. 
    Journal of Optimization Theory and Applications, 154(1), 235-257.
    \href{https://doi.org/10.1007/s10957-011-9983-3}{doi:10.1007/s10957-011-9983-3}


	
		\bibitem{AB}
		Ashbaugh, M. S., \& Benguria, R. D. (1993). Universal bounds for the low eigenvalues of Neumann Laplacians in $N$ dimensions. SIAM Journal on Mathematical Analysis, 24(3), 557-570.
		\href{https://doi.org/10.1137/0524034}{doi:10.1137/0524034}
		
		
		\bibitem{axler}
		Axler, S., Bourdon, P., \& Wade, R. (2013). Harmonic function theory (Vol. 137). Springer Science \& Business Media.
		\href{https://doi.org/10.1007/978-1-4757-8137-3}{doi:10.1007/978-1-4757-8137-3}
		
	\bibitem{BDP}
	Brasco, L., \& De Philippis, G. (2017). 7 Spectral inequalities in quantitative form. In Shape optimization and spectral theory (pp. 201-281). De Gruyter Open Poland.
	\href{https://doi.org/10.1515/9783110550887-007}{doi:10.1515/9783110550887-007}
		
	\bibitem{BP}	
	Brasco, L., \& Pratelli, A. (2012). Sharp stability of some spectral inequalities. Geometric and Functional Analysis, 22(1), 107-135.
	\href{https://doi.org/0.1007/s00039-012-0148-9}{doi:0.1007/s00039-012-0148-9}
		
	\bibitem{courant}
   	Courant, R., \& Hilbert, D. (1937). Methods of Mathematical Physics, Volume I. Wiley.
   	\href{https://doi.org/10.1002/9783527617210}{doi:10.1002/9783527617210}
		
		\bibitem{BH}
		Bucur, D., \& Henrot, A. (2019). Maximization of the second non-trivial Neumann eigenvalue. Acta Mathematica, 222(2), 337-361.
		\href{https://doi.org/10.4310/ACTA.2019.v222.n2.a2}{doi:10.4310/ACTA.2019.v222.n2.a2}

  		\bibitem{daners}
		Daners, D. (1999). Local singular variation of domain for semilinear elliptic equations. In: Escher J., Simonett G. (eds) Topics in Nonlinear Analysis. Progress in Nonlinear Differential Equations and Their Applications (pp. 117-141), Birkh\"auser, Basel.
		\href{https://doi.org/10.1007/978-3-0348-8765-6\_8}{doi:10.1007/978-3-0348-8765-6\_8}
  		
  		
		\bibitem{EF1}
		Enache, C., \& Philippin, G. A. (2013). Some inequalities involving eigenvalues of the Neumann Laplacian. Mathematical Methods in the Applied Sciences, 36(16), 2145-2153.
		\href{https://doi.org/10.1002/mma.2743}{doi:10.1002/mma.2743}
		
		\bibitem{EF2}
		Enache, C., \& Philippin, G. A. (2015). On some isoperimetric inequalities involving eigenvalues of symmetric free membranes. ZAMM Zeitschrift f\"ur Angewandte Mathematik und Mechanik, 95(4), 424-430.
		\href{https://doi.org/10.1002/zamm.201300211}{doi:10.1002/zamm.201300211}
		
		\bibitem{EP3}
		Enache, C., \& Philippin, G. A. (2017). On some inequalities for low eigenvalues of closed surfaces in. Applicable Analysis, 96(15), 2516-2525.
		\href{https://doi.org/10.1080/00036811.2016.1227968}{doi:10.1080/00036811.2016.1227968}
		
		\bibitem{GNP}
		Girouard, A., Nadirashvili, N., \& Polterovich, I. (2009). Maximization of the second positive Neumann eigenvalue for planar domains. Journal of Differential Geometry, 83(3), 637-662.
		\href{https://doi.org/10.4310/jdg/1264601037}{doi:10.4310/jdg/1264601037}
		
		\bibitem{HS}
		Helffer, B., \& Sundqvist, M. (2016). On nodal domains in Euclidean balls. Proceedings of the American Mathematical Society, 144(11), 4777-4791.
		\href{https://doi.org/10.1090/proc/13098 }{doi:10.1090/proc/13098}
	
		\bibitem{hersh2}
		Hersch, J. (1965). On symmetric membranes and conformal radius: Some complements to P\'olya's and Szeg\H{o}'s inequalities. 
		Archive for Rational Mechanics and Analysis, 20(5), 378-390.
		\href{https://doi.org/10.1007/BF00282359}{doi:10.1007/BF00282359}
		
		\bibitem{hersch3}
		Hersch, J. (1973). Lower bounds for membrane eigenvalues by anisotropic auxillary problems. Applicable Analysis, 3(3), 241–245. 
		\href{https://doi.org/10.1080/00036817308839068}{doi:10.1080/00036817308839068}
		
		\bibitem{IS}
		Ifantis, E. K., \& Siafarikas, P. D. (1985). A differential equation for the zeros of Bessel functions. Applicable Analysis, 20(3-4), 269-281.
		\href{https://doi.org/10.1080/00036818508839574}{doi:10.1080/00036818508839574}
		
		\bibitem{kennedy}
		Kennedy, J. B. (2018). A toy Neumann analogue of the nodal line conjecture. Archiv der Mathematik, 110(3), 261-271.
		\href{https://doi.org/10.1007/s00013-017-1117-1}{doi:10.1007/s00013-017-1117-1}
		

		\bibitem{KZ}
		Kong, Q., \& Zettl, A. (1996). Eigenvalues of regular Sturm-Liouville problems. Journal of Differential Equations, 131(1), 1-19.
		\href{https://doi.org/10.1006/jdeq.1996.0154}{doi:10.1006/jdeq.1996.0154}
		
		\bibitem{KS}
		Kornhauser, E. T., \& Stakgold, I. (1952). A variational theorem for $\nabla^2 u+ \lambda u= 0$ and its application. Journal of Mathematics and Physics, 31(1-4), 45-54.
		\href{https://doi.org/10.1002/sapm195231145}{doi:10.1002/sapm195231145}
    
    
    	\bibitem{kuttler}
    	Kuttler, J. R. (1984). A new method for calculating TE and TM cutoff frequencies of uniform waveguides with lunar or eccentric annular cross section. IEEE transactions on microwave theory and techniques, 32(4), 348-354.
    	\href{https://doi.org/10.1109/TMTT.1984.1132682}{doi:10.1109/TMTT.1984.1132682}
		
		\bibitem{Li}
		Li, L. (2007). On the second eigenvalue of the Laplacian in an annulus. Illinois Journal of Mathematics, 51(3), 913-925.
		\href{https://doi.org/10.1215/ijm/1258131110}{doi:10.1215/ijm/1258131110}
		
		\bibitem{LS}
		Lorch, L., \& Szego, P. (1994). Bounds and monotonicities for the zeros of derivatives of ultraspherical Bessel functions. SIAM Journal on Mathematical Analysis, 25(2), 549-554.
		\href{https://doi.org/10.1137/S0036141092231458}{doi:10.1137/S0036141092231458}
		
		\bibitem{nad1}
		Nadirashvili, N. (1996). Conformal maps and isoperimetric inequalities for eigenvalues of the Neumann problem. In Proceedings of the Ashkelon Workshop on Complex Function Theory (Vol.\ 11, pp.\ 197-201).
	
		\bibitem{RS}
	Radin, C., \& Sadun, L. (1998). Subgroups of $SO(3)$ associated with tilings. 
	Journal of Algebra, 202(2), 611-633.
	\href{https://doi.org/10.1006/jabr.1997.7320}{DOI:10.1006/jabr.1997.7320}	

		\bibitem{shubin}
		Shubin, M. A. (1987). Pseudodifferential operators and spectral theory (Vol. 200, No. 1). Berlin: Springer-Verlag.	
		\href{https://doi.org/10.1007/978-3-642-56579-3}{doi:10.1007/978-3-642-56579-3}
     
		\bibitem{szego}
		Szeg\"{o}, G. (1954). Inequalities for certain eigenvalues of a membrane of given area. Journal of Rational Mechanics and Analysis, 3, 343-356.
		\url{https://www.jstor.org/stable/24900293}
		    
		\bibitem{weinberger}
		Weinberger, H. F. (1956). An isoperimetric inequality for the $N$-dimensional free membrane problem. Journal of Rational Mechanics and Analysis, 5(4), 633-636.
		\url{https://www.jstor.org/stable/24900219}
		
		\bibitem{yee}
		Yee, H. Y., \& Audeh, N. F. (1966). Cutoff frequencies of eccentric waveguides. IEEE Transactions on Microwave Theory and Techniques, 14(10), 487-493.
		\href{https://doi.org/10.1109/TMTT.1966.1126308}{doi:10.1109/TMTT.1966.1126308}
		
	\end{thebibliography}
\end{document}